\documentclass[reqno]{amsart}
\usepackage{bbm}
\usepackage{amsmath,esint}
\usepackage[foot]{amsaddr}
\usepackage{amssymb}
\usepackage{mathtools}
\usepackage{bookmark}

\def\XXint#1#2#3{{\setbox0=\hbox{$#1{#2#3}{\int}$ }
\vcenter{\hbox{$#2#3$ }}\kern-.6\wd0}}

\DeclarePairedDelimiter\abs{\lvert}{\rvert}

 \allowdisplaybreaks
\def\R{\mathbb R}
\def\H{\mathcal H}

\def\Om{\Omega}
\def\N{\mathbb N}
\def\e{\varepsilon}

\newcommand{\D}{\nabla}
\newcommand{\dd}{\partial}
\newcommand{\norm}[1]{\left\lVert#1\right\rVert}

\newtheorem{theorem}{Theorem}[section]
\newtheorem{lemma}[theorem]{Lemma}
\newtheorem{remark}[theorem]{Remark}
\newtheorem{definition}[theorem]{Definition}

\newtheorem{Notation}[theorem]{Notation}

\begin{document}
\title
[Regularity Results for an Optimal Design Problem \dots]{Regularity Results for an Optimal Design Problem with lower order terms}


\author{L. Esposito}
\address{Dipartimento di Matematica, Via Giovanni Paolo II, 132, 84084 Fisciano (SA)}\email{luesposi@unisa.it}
\author{L. Lamberti}
\email{llamberti@unisa.it}

\begin{abstract}
We study the regularity of the interface for optimal energy configurations of functionals involving  bulk energies with an additional perimeter penalization of the interface. It is allowed the dependence on $(x,u)$ for the bulk energy. For a minimal configuration $(E,u)$, the H\"{o}lder continuity of $u$ is well known. We give an estimate for the singular set of the boundary $\partial E$. Namely we show that the Hausdorff dimension of the singular set is strictly smaller than $n-1$.
\end{abstract}

\maketitle


\bigskip

\subjclass{\noindent { 2010 Mathematics Subject Classification:} 49Q10, 49N60, 49Q20}

\section{Introduction and statements}
In this paper we will deal with energy functionals of the type
\begin{equation}\label{intro0}
{\mathcal F}(E,u;\Omega)=\int_\Om\ \bigl [ F(x,u,\nabla u)+\mathbbm{1}_{E}G(x,u,\nabla u)\bigr ]\,\,dx +P(E,\Om)\,,
\end{equation}
where $u\in H^1(\Omega)$ and $\mathbbm{1}_{E}$ denotes the characteristic function of a set $E\subset \Omega$ with finite perimeter in $\Omega$, denoted as $P(E,\Omega)$.
In mathematical and physical literature, the problem of finding the minimal energy configuration of a mixture of two materials in a bounded connected open set $\Omega \subset \mathbb{R}^n$ has been widely investigated (see for instance \cite{AB}, \cite{AC}, \cite{FFLM}, \cite{Gur}, \cite{Lar}, \cite{Lin}). The energy functional employed in such problems involves both bulk and interface energies in order to describe a large class of phenomenon in many applied sciences, such as non linear
elasticity, material sciences and image segmentations in the computer vision. A relevant case deeply studied by several authors  is the following model functional,
\begin{equation}\label{model}
\int_{\Omega}\sigma_{E}(x)|\nabla u|^2\,dx+ P(E,\Omega),
\end{equation}
where $u=u_0$ is prescribed on $\partial \Omega$ and $\sigma_{E}(x)=\beta \mathbbm{1}_{E}+\alpha \mathbbm{1}_{\Omega\setminus E}$, with $0<\alpha<\beta$ given constant.

In $1993$ L. Ambrosio and G. Buttazzo in \cite{AB} proved that, if $(E,u)$ is a minimizer of the functional $\eqref{model}$, then $u$ is locally H\"{o}lder continuous in $\Omega$ and $E$ is relatively open in $\Omega$. In the same year F.H. Lin  proved the regularity of the interface (see \cite{Lin}), that is $\partial E$ is regular outside a relatively closed set of vanishing $\mathcal{H}^{n-1}$-measure.
To be more precise, we define the set of regular points of $\partial E$ as follows:
\begin{equation}
\mbox{Reg}(E):= \left\{x\in \partial E\cap \Omega:\ \partial E \mbox{ is a }C^{1,\gamma} \mbox{ hypersurface in some } I(x) 
\right\}
\end{equation}
where $I(x)$ denotes a neighborhood of $x$. Accordingly, we define the set of singular points of $\partial E$
\begin{equation}
\Sigma(E) := (\partial E \cap \Omega)\setminus \mbox{Reg}(E).
\end{equation}
In the paper by F.H. Lin (see \cite{Lin}) it was proved that $\mathcal {H}^{n-1}(\Sigma(E))=0$ for minimal configurations of the functional $\eqref{model}$. More recently G. De Philippis and A. Figalli in \cite{DF}, Fusco and Julin in \cite{FJ}, independently and by different approaches, were able to improve this result indeed proving that
\begin{equation}
\label{redu}
dim_{\mathcal{H}}(\Sigma(E))\leq s,
\end{equation}
for some $s<n-1$ depending only on $\alpha, \beta$. Regarding this dependence, it is worth noticing that in \cite{EF} it has been proven, assuming $\alpha=\beta$ and $\beta\leq \gamma$ with $\gamma=\gamma(n)$, that $u\in C^{0,\frac{1}{2}+\varepsilon}$, $\partial^{*} E$ is a $C^{1,\varepsilon}-$hypersurface and $\mathcal{H}^s(\partial E\setminus \partial^* E)=0$ for all $s>n-8$.

In 1999 F.H. Lin and R.V. Kohn in [LK] treated a more general quadratic bulk energy of the type $\eqref{intro0}$ actually proving the same regularity they proved for the model case $\eqref{model}$. As a matter of fact they proved that the singular part $\Sigma(E)$ has vanishing $\mathcal{H}^{n-1}$-measure for minimal configurations $(E,u)$. 

In this paper we address the issue of improving the dimensional estimate for the singular part $\Sigma(E)$ of optimal configurations for the class of functionals treated by Lin and Kohn. As a matter of fact we prove, for a wide class of quadratic functional depending also on $x$ and $u$,
the same kind of regularity proved in the model case $\eqref{model}$ by [DF] and [FJ] namely, 
$ dim_{\mathcal{H}}(\Sigma(E))\leq s$ for some $s<n-1$.

Our path to prove the aforementioned result basically follows the same strategy used in [FJ]. As we will point out later in more detail, our technique relies on the linearity of the Euler equation of the functional $\eqref{intro0}$. For this reason we need a linear structure condition for the bulk energy. In the rest of the paper we will assume that the density energies $F$ and $G$ in $\eqref{intro0}$ satisfy the following assumptions:

\begin{eqnarray}\label{structure1}
&&F(x,s,z)=\sum_{i,j=1}^n a_{ij}(x,s)z_iz_j+\sum_{i=1}^n a_i(x,s)z_i+a(x,s),\\
\label{structure2}
&&G(x,s,z)=\sum_{i,j=1}^n b_{ij}(x,s)z_iz_j+\sum_{i=1}^n b_i(x,s)z_i+b(x,s).
\end{eqnarray}
Concerning the coefficients we assume that
\begin{equation*}
a_{ij},b_{ij},a_i,b_i,a,b\in C^{0,1}(\Omega\times\R).
\end{equation*} 
We will denote by $L_D$ the greatest lipschitz constant of the data 
$a_{ij},b_{ij}, a_i,b_i,a,b$, that is
\begin{equation}\label{Hoelderianity}
|\D a_{ij}|\leq L_D,\;\;|\D b_{ij}|\leq L_D \;\;\mbox{ in }\Omega\times\R,
\end{equation}
and the same holds true for $a_i,b_i,a,b$.\\
Moreover to ensure the existence of minimizers we assume the boundedness of the coefficients and the ellipticity of the matrices $a_{ij}$ and $b_{ij}$,
\begin{eqnarray}\label{ellipticity1}
&&\nu|\xi|^2\leq a_{ij}(x,s)\xi_i\xi_j\leq N |\xi|^2,\;\;\;\nu |\xi|^2\leq b_{ij}(x,s)\xi_i\xi_j\leq  N |\xi|^2\\
\label{ellipticity2}
&& \sum_{i=1}^n |a_i(x,s)|+\sum_{i=1}^n |b_i(x,s)|+|a(x,s)|+|b(x,s)|\leq L
\end{eqnarray}
We are interested in the regularity of minimizers of the following constrained problem.

\begin{definition} In the sequel we shall denote by $\eqref{P_c}$ the constrained problem
\begin{equation}
\label{P_c}
\min_{\substack{E\in \mathcal{A}(\Omega)\\v\in u_0+H_0^1(\Omega)}}
\left\{
\mathcal{F}(E,v;\Omega)\,:\,  |E|=d
\right\}
\tag{$P_c$}
\end{equation}
where $u_0\in H^1(\Omega)$, $0<d<|\Omega|$ are given and $\mathcal{A}(\Omega)$ is the class of all subsets of $\Omega$ with finite perimeter.
\end{definition}
The problem of handling with the constraint $|E|=d$ is overtaken using an argument introduced in \cite{EF}, ensuring that every minimizer of the constrained problem $\eqref{P_c}$ is also a minimizer of a penalized functional of the type
\begin{equation}\label{intro1}
{\mathcal F}_{\Lambda}(E,v;\Omega)=\mathcal{F}(E,v;\Omega)+\Lambda\big||E| -d\big|,
\end{equation}
for some suitable $\Lambda>0$ (see Theorem \ref{Teorema Penalizzazione} below). Therefore we give in addition the following definition.
\begin{definition}

In the sequel we shall denote by $\eqref{P}$ the penalized problem
\begin{equation}
\label{P}
\min_{\substack{E\in \mathcal{A}(\Omega)\\v\in u_0+H_0^1(\Omega)}}
\mathcal{F}_{\Lambda}(E,v;\Omega)
\tag{$P$}
\end{equation}
where $u_0\in H^1(\Omega)$, is given, and $\mathcal{A}(\Omega)$ is the class of all subsets of $\Omega$ with finite perimeter in $\Omega$.
\end{definition}
From the point of view of regularity, the extra term $\Lambda\big||E| -|F|\big|$ is a higher order negligible perturbation.
The main result of the paper is contained in the following theorem.
\begin{theorem}
\label{Teorema principale}
Let $(E,u)$ be a minimizer of either problem $\eqref{P_c}$ or problem $\eqref{P}$, under assumptions $\eqref{structure1}-\eqref{ellipticity2}$. Then
\begin{itemize}
\item[a)] there exists a relatively open set $\Gamma\subset \partial E$ such that $\Gamma$ is a $C^{1,\mu}$ hypersurface for all $0<\mu<\frac{1}{2}$,
\item[b)] there exists $\varepsilon >0$ depending on $\nu,N,L,n$, such that $$\mathcal{H}^{n-1-\varepsilon}((\partial E\setminus \Gamma)\cap \Omega)=0.$$
\end{itemize}
\end{theorem}
For reader convenience the paper is structured in sections which reflect the proof strategy. Section $2$ collects known results and preliminary definitions. As in the case of minimizers of the Mumford-Shah functional the proof of regularity is based on the study of interplay between the perimeter and the bulk energy.
We point out that the H\"older exponent $\frac{1}{2}$ is critical for solutions $u$ of either $\eqref{P}$ or $\eqref{P_c}$, in the sense that whenever $u\in C^{0,\frac{1}{2}}$, under appropriate scaling the bulk term  locally have the same dimension $n-1$ as the perimeter term.
In this regard, our starting point is to prove suitable energy decay estimates for the bulk energy. These estimates are contained in section 3. The key point of this approach is contained in Lemma \ref{Lemma decadimento 1}, where it is proved that the bulk energy decay faster than $\rho^{n-1}$, that is, for any $\delta>0$,
\begin{equation}\label{intro2}
\int_{B_{\rho}(x_0)}|\D u|^2\,dx\leq C \rho^{n-\delta},
\end{equation} 
either in the case that 
$$\min\{|E\cap B_{\rho}(x_0)|,|B_{\rho}(x_0)\setminus E|\}<\varepsilon_0|B_{\rho}|,
$$
or in the case that, there exists an half space $H$ such that
$$
|(E\Delta H)\cap B_{\rho}(x_0)|\leq \varepsilon_0|B_{\rho}|.
$$
The latter case is the hardest to handle because it relies on the regularity properties of solutions of a transmission problem which we study in subsection 3.1. 
Let us notice that, for any given $E\subset \Omega$, local minimizers $u$ of the functional
\begin{equation}\label{intro3}
\int_\Om\ \bigl [ F(x,u,\nabla u)+\mathbbm{1}_{E}G(x,u,\nabla u)\bigr ]\,dx
\end{equation}
are H\"older continuous, $u\in C^{0,\alpha}_{loc}(\Omega)$, but the needed bound $\alpha>\frac{1}{2}$ cannot be expected in the general case without any information on the set $E$. In section 3.1 we prove, in the case $E$ is an half space, that
minimizers of the functional \eqref{intro3} are in $C^{0,\alpha}$ for every $\alpha>0$. In this context the linearity of the equation strongly come in to play ensuring that the derivatives of the Euler equation are again solutions of the same equation. For the proof in section $3$ we readapt a technique depicted in the book \cite{AFP} in the context of the Mumford and Shah functional and recently used in a paper by E. Mukoseeva and G. Vescovo, \cite{MV}. 
Once the estimates of section $3$ are obtained, we are ready to prove in section $4$ that, if in a ball $B_{\rho}(x_0)$ the  perimeter of $E$ is sufficiently small, then the total energy 
$$
\int_{B_r(x_0)}|\D u|^2\,dx + P(E,B_{r}(x_0)), \quad \quad 0<r<\rho,
$$
decays as $r^{n}$ (see Lemma \ref{Lemma decadimento 2}). Making use of the latter energy density estimate we are in position to deduce in section $4$ the density lower bound for the perimeter of $E$ as well. In the subsequent sections the proof strategy follows the path traced from the regularity theory for perimeter minimizers. In section 5 it is proved the compactness for sequences of minimizers which more or less follows in a standard way from the density lower bound. Section 6 is devoted to prove some additional consequences of the density lower bound which involve the excess
$$
{\mathbf e}(x,r)=\inf_{\nu\in \mathbb{S}^{n-1}}{\mathbf e}(x,r,\nu):= \inf_{\nu\in \mathbb{S}^{n-1}}\frac{1}{r^{n-1}}\int_{\partial E\cap B_r(x)}\frac{|\nu_E(y)-\nu|^2}{2}d\mathcal H^{n-1}(y).
$$
Actually we prove the height bound Lemma and the Lipschitz approximation Theorem. In this section we also compute the Euler equation for $\mathcal{F}(E,u)$ involving the variation of the set $E$. Section 7 is devoted to prove the excess improvement which follows from the fact that whenever the excess ${\mathbf e}(x,r)$ goes to zero for $r\rightarrow 0$ the Dirichlet integral $\int_{B_{\rho}(x_0)}|\D u|^2\,dx$ decays as in \eqref{intro2}. In section 8 we give the proof of Theorem \ref{Teorema principale} that is a consequence of the excess improvement proved before.
\section{Preliminary notations and definitions}
In the rest of the paper we will write $\langle \xi, \eta \rangle$ for the inner product of vectors $\xi, \eta \in \mathbb{R}^n$, and $|\xi|:=\langle \xi, \xi \rangle^{\frac 12}$ will denote the corresponding Euclidean norm.
If $u$ is integrable in $B_R(x_0)$ we set:
$$
u_R=\frac{1}{\omega_n R^n}\int_{B_R(x_0)} u\,dx = \fint_{B_{R}(x_0)} u\,dx
$$
The following definition is standard.
\begin{definition}
Let $v\in H^1_{loc}(\Omega)$, and assume that $E\subset \Omega$ is fixed. We say that $v$ is a local minimizer of the integral functional $\mathcal{F}$ defined in \eqref{intro0} iff
$$
{\mathcal F}(E,v;B_R(x_0))=\min \left\{{\mathcal F}(E,w;B_R(x_0)):\,\, w\in v+H^1_0(B_R(x_0))
\right\}
$$
for all $B_R(x_0)\subset \subset \Omega$.
\end{definition}
It is clear that any minimizer $u$ of problem $\eqref{P_c}$ is a local minimizer of the functional \eqref{intro0} and therefore satisfies the Euler equation
\begin{eqnarray}\label{E}
&&\sum_{i=1}^{n}
\frac{\partial}{\partial{x_i}}
\Big [ F_{z_i}(x,u,\nabla u)+\mathbbm{1}_{E}G_{z_i}(x,u,\nabla u)\bigr)\Big ]\\
\nonumber && =F_u(x,u,\nabla u)+G_u(x,u,\nabla u)
\end{eqnarray}
\begin{lemma}
\label{Lemma iterativo 1}
Let $Z(t)$ be a bounded non-negative function in the interval $[\rho,R]$ and assume that for $\rho\leq t<s\leq R$ we have
\begin{equation}\label{Z1}
Z(t)\leq \theta Z(s)+\frac{A}{(s-t)^2}+B,
\end{equation}
with $A,B\geq 0$, and $0\leq\theta <1$. Then
\begin{equation}\label{Z2}
Z(\rho)\leq c(\theta)\Bigl[\frac{A}{(R-\rho)^2}+B\Bigr],
\end{equation}
for some $c(\theta)$ depending only on $\theta$.
\end{lemma}
\begin{proof}
The proof of this lemma is standard and can be found in \cite{Giu}. Inspecting the proof in \cite{Giu} one can also obtain an explicit expression of the constant $c(\theta)$. An admissible value for $c(\theta)$, but not the best, is  $c(\theta)=1/(1-\theta^{1/3})^3$.
\end{proof}
The next lemma can be found in \cite[Lemma 7.54]{AFP}.
\begin{lemma}
\label{Lemma iterativo 2}
Let $f:(0,a]\rightarrow [0,\infty)$ be an increasing function such that
$$
f(\rho)\leq A\Bigl[\Bigl(\frac{\rho}{R}\Bigr)^p+R^s\Bigr]f(R)+BR^q\quad\text{whenever }0<\rho<R\leq a
$$
for some constants $A,B\geq 0$, $0<q<p$, $s>0$. Then there exist $R_0(p,q,s,A)$ and $c(p,q,A)$ such that
$$
f(\rho)\leq c\Bigl(\frac{\rho}{R}\Bigr)^qf(R)+ cB\rho^q \qquad\mbox{whenever }0<\rho<R\leq \min\{R_0,a\}.
$$
\end{lemma}
\subsection{From constrained to penalized problem}
The next theorem allows to overcome the difficulty of handling with the constraint $|E|=d$. As a matter of fact it can be proved that every minimizer of the constrained problem $\eqref{P_c}$ is also a minimizer of a suitable unconstrained problem with a volume penalization of the type $\eqref{P}$.
\begin{theorem}
\label{Teorema Penalizzazione} Let $0<d<|\Omega|$. There exists $\Lambda_{0}>0$ such that if $(F,v)$ is a minimizer of the functional
\begin{equation}
\label{Penalized}
{\mathcal F}_{\Lambda}(A,w)=\int_\Om\ \Bigl ( F(x,u,\nabla u)+\mathbbm{1}_{A}G(x,u,\nabla u)\,dx\Bigr )\,dx +P(A,\Om)+\Lambda\big||A| - d\big|
\end{equation}
for some $\Lambda \geq \Lambda_0$, among all configurations $(A,w)$ such that $w=v_0$ on $\partial \Omega$,
then $|F|=d$ and $(F,v)$ is a minimizer of problem $\eqref{P_c}$.
Conversely, if $(E,u)$ is minimizer of problem $\eqref{P_c}$, then it is a minimizer of \eqref{Penalized}, for all $\Lambda \geq \Lambda_0$.
\end{theorem}
\begin{proof}
The proof can be carried out as \cite[Theorem 1]{EF}. For the convenience of the reader we give here its sketch, emphasizing main ideas and minor differences with respect to the case treated in \cite{EF}.\\
The first part of the theorem can be proved by contradiction. Assume that there exists a sequence $(\lambda_h)_{h\in \mathbb N}$ such that $\lambda_h\rightarrow +\infty$ as $h\rightarrow +\infty$ and a sequence of configurations $(E,u_h)$ minimizing $\mathcal{F}_{\lambda_h}$ and such that $u_h=u_0$ on $\partial \Omega$ and $|E_h|\neq d$ for all $h$. Le us choose now an arbitrary fixed $E_0\subset \Omega$ with finite perimeter and such that $|E|=d$. Let us point out that 
\begin{equation}\label{Theta}
\mathcal{F}_{\lambda_h}(E_h,u_h)\leq\mathcal{F}(E_0,u_0):=\Theta.
\end{equation}
Without loss of generality we can assume that $|E_h|<d$. As a matter of fact the case $|E_h|>d$ can be treated in the same way considering the complement of $E$ in $\Omega$. Our aim is to show that for $h$ sufficiently large, there exists a configuration $(\widetilde{E}_h,\tilde{u}_h)$ such that $\mathcal{F}_{\lambda_h}(\widetilde{E}_h,\tilde{u}_h)< \mathcal{F}_{\lambda_h}({E_h},{u_h})$, thus proving the result by contradiction.\\
By condition $\eqref{Theta}$, it follows that the sequence $\{u_h\}$ is bounded in $H^1(\Omega)$, the perimeters of the sets $E_h$ in $\Omega$ are bounded and $|E_h|\rightarrow d$. Therefore, possibly extracting a not relabelled subsequence, we may assume that there exists a configuration $(E,u)$ such that $u_h\rightarrow u$ weakly in $H^1(\Omega)$, $\mathbbm{1}_{E_h}\rightarrow \mathbbm{1}_{E}$ a.e. in $\Omega$, where the set $E$ is of finite perimeter in $\Omega$ and $|E|=d$. The couple $(E,u)$ will be used as reference configuration for the definition of $(\widetilde{E}_h,\tilde{u}_h)$.\\

{\bf Step 1.} {\em Construction of $(\widetilde{E}_h,\tilde{u}_h)$}.
Proceeding exactly as in \cite{EF}, we take a point $x\in\partial^*E\cap\Om$ and observe that  the sets $E_r=(E-x)/r$ converge locally in measure to the half space $H=\{z\cdot\nu_E(x)>0\}$, i.e., $\chi_{E_r}\to\chi_H$ in $L^1_{\rm loc}(\R^n)$, where $\nu_E(x)$ is the generalized inner normal to $E$ at $x$ (see \cite[Definition 3.54]{AFP}). Let $y\in B_1(0)\setminus H$ be the point $y=-\nu_E(x)/2$. Given $\e$ (to be chosen at the end of the proof), since $\chi_{E_r}\to\chi_H$ in $L^1(B_1(0))$ there exists $0<r<1$ such that
$$
|E_r\cap B_{1/2}(y)|<\e,\qquad |E_r\cap B_1(y)|\geq|E_r\cap B_{1/2}(0)|>\frac{\omega_n}{2^{n+2}}\,,
$$
where $\omega_n$ denotes the measure of the unit ball of $\R^n$ and $x_r=x+ry\in\Om$. Therefore we have
$$
|E\cap B_{r/2}(x_r)|<\e r^n,\qquad|E\cap B_r(x_r)|>\frac{\omega_nr^n}{2^{n+2}}\,.
$$
Let us assume, without loss of generality, that $x_r=0$ and from now on let us denote the balls centered at the origin by $B_r$. From the convergence of $E_h$ to $E$ we have that for all $h$ sufficiently large
\begin{equation}\label{unodue}
|E_h\cap B_{r/2}|<\e r^n,\qquad|E_h\cap B_r|>\frac{\omega_nr^n}{2^{n+2}}\,.
\end{equation}
Let us now define the following bi-Lipschitz map used in \cite{EF} which maps $B_r$ into itself 
\begin{equation}\label{unotre}
\Phi(x)=
\begin{cases}
\bigl(1-\sigma(2^n-1)\bigr)x & \text{if}\,\,\,|x|<\displaystyle\frac{r}{2},\\
x+\sigma\Bigl(\displaystyle1-\frac{r^n}{|x|^n}\Bigr)x & \text{if}\,\,\,\displaystyle\frac{r}{2}\leq|x|<r,\\
x & \text{if}\,\,\,|x|\geq r\,,
\end{cases}
\end{equation}
for some fixed $0<\sigma<1/2^n$ such that, setting
$$
{\widetilde E}_h=\Phi(E_h),\qquad{\tilde u}_h=u_h\circ\Phi^{-1}\,,
$$
we have $|{\widetilde E}_h|<d$. We obtain
\begin{eqnarray}\label{unoquattro}
&&{\mathcal F}_{\lambda_h}(u_h,E_h)-{\mathcal F}_{\lambda_h}({\tilde u}_h,{\widetilde E}_h)
\nonumber =\biggl[\int_{B_r}\big[F(x,u_h,\nabla u_h)+\mathbbm{1}_{E_h}G(x,u_h,\nabla u_h)\big]\,dx\\
&&-\int_{B_r}\big[F(x,\tilde{u}_h,\nabla \tilde{u}_h)+\mathbbm{1}_{\widetilde{E}_h}G(x,\tilde{u}_h,\nabla \tilde{u}_h)\big]\,dx\biggr] 
\\
&& \quad+\bigl[P(E_h,{\overline B}_r)-P({\widetilde E}_h,{\overline B}_r)\bigr]+\lambda_h\bigl(|{\widetilde E}_h|-|E_h|\bigr) \nonumber \\
&&= I_{1,h}+I_{2,h}+I_{3,h}\,. \nonumber
\end{eqnarray}
{\bf Step 2.} {\em Estimate of $I_{1,h}$}. In order to estimate the bulk energy $I_{1,h}$ we write down two preliminary estimates for the map $\Phi$ that can be easily obtained by direct computation (see \cite{EF} for the explicit calculation). There exists a constant $C_1=C_1(n)$ depending only on $n$ such that, for $\sigma<1/4^n(n-1)^2-1$,
\begin{eqnarray}\label{phi}
\bigl\|\nabla\Phi^{-1}\bigl(\Phi(x)\bigr)\bigr\|\leq\bigl(1-(2^n-1)\sigma\bigr)^{-1}\leq1+2^nn\sigma \qquad\text{for all}\,\,\,&&x\in B_r,\\
\label{Jphi}
1+C_1(n)\sigma\leq J\Phi(x)\leq1+2^nn\sigma\qquad\text{for all}\,\,\,&&x\in B_r.
\end{eqnarray}
We can now perform the change of variables $y=\Phi(x)$, and observing that $\mathbbm{1}_{\widetilde{E}_h}(\Phi(x))=\mathbbm{1}_{E_h}(x)$, we get
\begin{eqnarray}\nonumber
I_{1,h}=\int_{B_r}\bigl[F(x,u_h,\nabla u_h) - J\Phi(x) F(\Phi(x),u_h(x),\nabla u_h(x)\circ \nabla \Phi^{-1}(\Phi(x)))\bigr]\,dx  &&\\  
+\int_{B_r\cap E_h}\bigl[G(x,u_h,\nabla u_h) - J\Phi(x) G(\Phi(x),u_h(x),\nabla u_h(x)\circ \nabla \Phi^{-1}(\Phi(x)))\bigr]\,dx &&\\\nonumber
:=J_{1,h}+J_{2,h}  \qquad\qquad\qquad\qquad\qquad\qquad\qquad\qquad\qquad\qquad\qquad\qquad\qquad&& 
\end{eqnarray}
The two terms $J_{1,h}$ and $J_{2,h}$, involving $F$ and $G$ in $B_r$ and $B_r\cap E_h$ respectively, can be treated in the same way. Therefore we just perform the calculation for $J_{1,h}$.\\ To make the exposition more clear when using the structure conditions $\eqref{structure1}$ and $\eqref{structure2}$ we introduce the following notations. $A_2(x,s)$ denotes the quadratic form and $A_1(x,s)$ denotes the linear form defined as follows
\begin{eqnarray*}
&&A_2(x,s)[\xi]:=a_{ij}(x,s)\xi_i\xi_j \qquad\text{for all}\,\,\,\xi\in \mathbb{R}^n,\\
&&A_1(x,s)[\xi]:=a_{i}(x,s)\xi_i \qquad\text{for all}\,\,\,\xi\in \mathbb{R}^n,
\end{eqnarray*}
analogously $A_0(x,s)=a(x,s)$. Accordingly we can write down
\begin{eqnarray}\label{J1h}\nonumber
J_{1,h}=\qquad\qquad\qquad\qquad\qquad\qquad\qquad\qquad\qquad\qquad\qquad\qquad\qquad\qquad\qquad\quad&&\\
\nonumber\int_{B_r} \!\Bigl\{A_2(x,u_h(x))[\nabla u_h(x)]\!-\!A_2(\Phi(x),u_h(x))[\nabla u_h(x)\!\circ\! \nabla \Phi^{-1}(\Phi(x))]J\Phi(x)\Big\}\,dx &&\\
\nonumber
+\!\int_{B_r} \!\Bigl\{A_1(x,u_h(x))[\nabla u_h(x)]\!-\!A_1(\Phi(x),u_h(x))[\nabla u_h(x)\!\circ\! \nabla \Phi^{-1}(\Phi(x))]J\Phi(x)\Big\}\,dx &&\\
+\!\int_{B_r}\!\Bigl\{A_0(x,u_h(x))\!-\!A_0(\Phi(x),u_h(x))J\Phi(x)\Big\}\,dx\qquad\qquad\qquad\qquad\qquad\qquad\qquad
\end{eqnarray}
We proceed estimating the first difference in the previous inequality, being the other similar and indeed easier to handle.
\begin{eqnarray}
\nonumber\int_{B_r} \!\Bigl\{A_2(x,u_h(x))[\nabla u_h(x)]\!-\!A_2(\Phi(x),u_h(x))[\nabla u_h(x)\!\circ\! \nabla \Phi^{-1}(\Phi(x))]J\Phi(x)\Big\}\,dx &&\\
\nonumber
=\int_{B_r} \!\Bigl\{A_2(\Phi(x),u_h(x))[\nabla u_h(x)]\!-\!A_2(\Phi(x),u_h(x))[\nabla u_h(x)\!\circ\! \nabla \Phi^{-1}(\Phi(x))]J\Phi(x)\Big\}\,dx &&\\
\nonumber
+\int_{B_r} \!\Bigl\{A_2(x,u_h(x))[\nabla u_h(x)]\!-\!A_2(\Phi(x),u_h(x))[\nabla u_h(x)]\Big\}\,dx = H_{1,h}+ H_{2,h}.&&
\end{eqnarray}
The first term $H_{1,h}$ can be estimated observing that, in consequence of  $\eqref{ellipticity1}$ we have, 
$$
|A_2[\xi]-A_2[\eta]|\leq N|\xi+\eta||\xi-\eta|\qquad\qquad \forall\xi,\eta\in \mathbb{R}^n.
$$
We apply the last inequality to the vectors
$$
\xi:=\nabla u_h(x), \qquad \qquad \eta:=\sqrt{J\Phi(x)}[\nabla u_h(x)\!\circ\! \nabla \Phi^{-1}(\Phi(x))],
$$
and use estimates $\eqref{phi}$ and $\eqref{Jphi}$ to get-
$$
|\xi-\eta|\leq \sigma C(n)|\nabla u_h(x)|\qquad \qquad|\xi+\eta|\leq C(n)|\nabla u_h(x)|,
$$
for some constant $C(n)$ depending only on $n$.
From the previous estimates we deduce that
\begin{equation}\label{H1h}
|H_{1,h}|\leq \sigma N C^2(n) \int_{B_r}|\nabla u_h(x)|^2\,dx\leq  \sigma N C^2(n)\Theta.
\end{equation}
The second term $H_{2,h}$ can be estimated using the Lipschitz assumption of $a_{i,j}$ and observing that $|x-\Phi(x)|\leq \sigma r2^n$. As a matter of fact we deduce that
\begin{equation}\label{H2h}
|H_{2,h}|\leq \sigma r2^n \sup_{i,j}[a_{i,j}]_{0,1} \int_{B_r}|\nabla u_h(x)|^2\,dx\leq  \sigma  C(n,[a_{i,j}]_{0,1})\Theta.
\end{equation}
In conclusion, since the other terms in $\eqref{J1h}$ can be estimated in the same way,
collecting estimates $\eqref{H1h}$ and $\eqref{H2h}$ we get
$$
|J_{1,h}|\leq \sigma  C(n,N,L_D)\Theta.
$$
The same estimate holds true for $J_{2,h}$ then we conclude that
\begin{equation}
\label{I1h}
I_{1,h}\geq -\sigma  C_2(n,N,L_D)\Theta
\end{equation}
for some constant $C_2(n,N,[a_{i,j}]_{0,1})$ depending on $n,N,\sup_{i,j}[a_{i,j}]_{0,1}$.\\
{\bf Step 3.} {\em Estimate of $I_{2,h}$}. In order to estimate $I_{2,h}$ we can use the area formula for maps between rectifiable set. If we denote by $T_{h,x}$ the tangential gradient of $\Phi$ along the approximate tangent space to $\partial^* E_h$ in $x$ and $T^*_{h,x}$ is the adjoint of the map $T_{h,x}$, the $(n-1)$-dimensional jacobian of $T_{h,x}$ is given by
$$
J_{n-1}T_{h,x}=\sqrt{{\rm det}\bigl(T^*_{h,x}\circ T_{h,x}\bigr)}.
$$
Thereafter we can  estimate
\begin{equation}\label{TJ}
J_{n-1}T_{h,x}\leq 1+\sigma+2^n(n-1)\sigma.
\end{equation}
We address the reader to \cite{EF} where explicit calculations are given.
To estimate $I_{2,h}$, we use the area formula for maps between rectifiable sets (\cite[Theorem~2.91]{AFP}), thus getting
\begin{eqnarray*}
I_{2,h}\!\!\!&=&\!\!\!P(E_h,{\overline B}_r)-P({\widetilde E}_h,{\overline B}_r)=\int_{\partial^*E_h\cap{\overline B}_r}\!d\H^{n-1}-\int_{\partial^*E_h\cap{\overline B}_r}\!J_{n-1}T_{h,x}\,d\H^{n-1} \\
\!\!\!&=&\!\!\!\int_{\partial^*E_h\cap{\overline B}_r\setminus B_{r/2}}\!\left(1-J_{n-1}T_{h,x}\right)\,d\H^{n-1}+\int_{\partial^*E_h\cap B_{r/2}}\!\left(1-J_{n-1}T_{h,x}\right)\,d\H^{n-1}\,.
\end{eqnarray*}
Notice that the last integral in the above formula is non-negative since $\Phi$ is a contraction in $B_{r/2}$, hence $J_{n-1}T_{h,x}<1$ in $B_{r/2}$, while from \eqref{TJ} we have
$$
\int_{\partial^*E_h\cap{\overline B}_r\setminus B_{r/2}}\!\left(1-J_{n-1}T_{h,x}\right)\,d\H^{n-1}\geq-2^nnP(E_h,{\overline B}_r)\sigma\geq-2^nn\Theta\sigma\,,
$$
thus concluding that
\begin{equation}\label{I2h}
I_{2,h}\geq-2^nn\Theta\sigma\,.
\end{equation}
{\bf Step 4.} {\em Estimate of $I_{3,h}$}.
To estimate $I_{3,h}$ we recall \eqref{unodue}, \eqref{unotre}, \eqref{Jphi}, thus getting
\begin{eqnarray*}
I_{3,h}\!\!\!&=&\!\!\! \lambda_h\int_{E_h\cap B_r\setminus B_{r/2}}\!\left(J\Phi(x)-1\right)\,dx+\lambda_h\int_{E_h\cap B_{r/2}}\!\left(J\Phi(x)-1\right)\,dx \\
\!\!\!&\geq&\!\!\! \lambda_hC_1(n)\Bigl(\frac{\omega_n}{2^{n+2}}-\e\Bigr)\sigma r^n-\lambda_h\bigl[1-\bigl(1-(2^n-1)\sigma\bigr)^n\bigr]\e r^n \\
\!\!\!&\geq&\!\!\!\lambda_h\sigma r^n\Bigl[C_1(n)\frac{\omega_n}{2^{n+2}}-C_1(n)\e-(2^n-1)n\e\Bigr]\,.
\end{eqnarray*}
Therefore, if we choose $0<\e<\e(n)$, with $\e(n)$ depending only on the dimension, we have that
$$
I_{3,h}\geq\lambda_hC_3(n)\sigma r^n\,,
$$
for some positive $C_3(n)$. From this inequality, recalling \eqref{unoquattro}, \eqref{I1h} and \eqref{I2h} we obtain
$$
{\mathcal F}_{\lambda_h}(u_h,E_h)-{\mathcal F}_{\lambda_h}({\tilde u}_h,{\widetilde E}_h)\geq\sigma\bigl(\lambda_hC_3r^n-\Theta(C_2(n,N,L_D)+2^nn)\bigr)>0
$$
if $\lambda_h$ is sufficiently large. This contradicts the minimality of $(u_h,E_h)$, thus concluding the proof.

\end{proof}

The previous theorem  motivates the following definition.
\begin{definition}[$\Lambda$-minimizers]
The energy pair $(E,u)$ is a $\Lambda$-minimizer in $\Omega$ of the functional $\mathcal {F}$, defined in \eqref{intro0}, iff for every $B_r(x_0)\subset \Omega$ 
$$
\mathcal{F}(E,u;B_r(x_0))\leq\mathcal{F}(F,v;B_r(x_0))+\Lambda |F\Delta E|,
$$
whenever $(F,v)$ is an admissible test pair, namely, $F$ is a set of finite perimeter with $F\Delta E\subset \subset B_r(x_0)$ and $v-u\in H^1_0(B_r(x_0))$.
\end{definition}

\section{Decay of the bulk energy}

In the first part of this section we collect some preliminary results concerning the decay estimates for local minimizers $u$ of the functional $\eqref{intro0}$ when $E$ is fixed. 
We start quoting higher integrability results both for local minimizers of functional $\eqref{intro0}$ and for comparison functions that we will use later in the paper.
It is worth mentioning that the following lemmata can be applyed in general to minimizers of integrals functionals of the type
\begin{equation}\label{H}
\mathcal{H}(u;\Omega):=\int_{\Omega}F(x,u,\D u)\,dx,
\end{equation}
assuming that the energy density only satisfies the structure condition $\eqref{structure1}$ and the growth conditions $\eqref{ellipticity1}$ and $\eqref{ellipticity2}$, without assuming any continuity on the coefficients. Therefore functionals of the type $\eqref{intro0}$ belong to this class and in addition the involved estimates only depend on the constants appearing in $\eqref{ellipticity1}$ and $\eqref{ellipticity2}$ but don't depend on $E$ accordingly.

\begin{lemma}Let $u\in H^1(\Omega)$ be a local minimizer of the functional $\mathcal{H}$ defined in  $\eqref{H}$, where $F$ satisfies the structure condition $\eqref{structure1}$ and the growth conditions $\eqref{ellipticity1}$ and $\eqref{ellipticity2}$. Then for every $B_{2R}(x_0)\subset \subset \Omega$ it holds
\begin{equation}\label{ReverseH}
\fint_{B_{R}(x_0)}|\nabla u|^2\,dx\leq C_1\Bigl( 1+\fint_{B_{2R}(x_0)}|\nabla u|^{2m}\,dx\Bigr)^\frac 1m,
\end{equation}
where $m=\frac{n}{n+2}$, $C_1=C_1(\nu,N,L,n)$ is a constant depending only on $\nu,N,L,n$.
\end{lemma}
\begin{proof}
Whithout loss of generality we can assume that $x_0=0$. Let  $R<t<s<2R$ and  choose $\eta\in C^{\infty}_0(Q_{s})$ such that $\eta\equiv 1 $ in $Q_t$ and $|\nabla \eta|\leq 2/(s-t)$. We choose a test function $v=u-\phi$, where $\phi=\eta(u-u_s)$ and $u_s$ denotes the average of $u$ in $Q_s$, $u_s=\fint_{Q_s} u\,dx$. From the growth conditions $\eqref{ellipticity1}$ and $\eqref{ellipticity2}$ we can deduce that 
\begin{equation}
\frac{\nu}{2}|z|^2-\frac{L^2}{\nu}\leq F(x,s,z)\leq (N+1)|z|^2+L(L+1),
\end{equation}
then testing the minimality of $u$ with $v$ we deduce that
\begin{eqnarray*}
&&\frac{\nu}{2}\int_{Q_s}[|\nabla u|^2 - \frac{L^2}{\nu^2}]\,dx\\
&&\leq 2\int_{Q_s}\Bigr [(N+1)|\nabla u (1-\eta)-\nabla \eta(u-u_s)|^2+
 L(L+1)\Bigr]\,dx\\
&&\leq 4(N+1)\int_{Q_s\setminus Q_t}|\nabla u|^2\,dx+4(N+1)\int_{Q_s}|u-u_s|^2|\nabla \eta|^2\,dx+2\int_{Q_s}L(L+1)\,dx.
\end{eqnarray*}
Adding to both sides $4(N+1)\int_{Q_t}|\nabla u|^2\,dx$ we deduce
\begin{eqnarray*}
&&[4(N+1)+\nu/2]\int_{Q_t}|\nabla u|^2\,dx\\
&&\leq 4(N+1)\int_{Q_s}|\nabla u|^2\,dx +4(N+1)\int_{Q_s}|u-u_s|^2|\nabla \eta|^2\,dx+\int_{Q_s}C(L,\nu)\,dx.
\end{eqnarray*}
Eventually we get
\begin{equation*}
\int_{Q_t}|\nabla u|^2\,dx
\leq \theta\int_{Q_s}|\nabla u|^2\,dx +\frac{C(\nu,N,L,n)}{(s-t)^2}\int_{Q_s}|u-u_s|^2\,dx+C(\nu,N,L,n).
\end{equation*}
Where $\theta=4(N+1)/4(N+1+\nu/2)<1$ and the constant $C(\nu,N,L,n)$ depends only on $\nu,N,L,n$.
We can iterate the previous estimate using \ref{Lemma iterativo 1} to deduce that there exist $C=C(\theta)=C(\nu,N)$ depending only on $\nu,N$ such that
\begin{equation*}
\int_{Q_R}|\nabla u|^2\,dx
\leq C(\nu,N)\biggl[\frac{C(\nu,N,L,n)}{(s-t)^2}\int_{Q_{2R}}|u-u_s|^2\,dx+C(\nu,N,L,n)\biggr].
\end{equation*}
Finally we use Sobolev-Poincar\'{e} inequality
$$
\int_{B_{2R}}|u-u_{2R}|^2\,dx \leq C(n)\Bigl(\int_{B_{2R}}|\nabla u|^{2m} \Bigr)^{\frac 1m},
$$
to conlude $\eqref{ReverseH}$.
\end{proof}
Starting from the previous lemma the higher integrability can be obtained in a standard way by means of the Gehring's Lemma (see \cite[Proposition 6.1]{Giu}).
\begin{lemma}
\label{Lemma maggiore sommabilità}
Let $u\in H^1(\Omega)$ be a local minimizer of the functional $\mathcal{H}$ defined in \eqref{H}, where $F$ satisfies the structure condition $\eqref{structure1}$ and the growth conditions $\eqref{ellipticity1}$ and $\eqref{ellipticity2}$. There exists an $s>1$ such that, for every ball $B_{2R}(x_0)\subset \subset \Omega$ it holds
\begin{equation*}
\fint_{B_{R}(x_0)}|\nabla u|^{2s}\,dx\leq C_2\Bigl( \fint_{B_{2R}(x_0)}(1+|\nabla u|^{2})\,dx\Bigr)^s,
\end{equation*}
where $s$ and $C_2$ depend only on $\nu,N,L,n$.
\end{lemma}
In the next section we will prove some energy  density estimates by using a standard comparison argument. For this purpose we will need a reverse H\"{o}lder inequality for the comparison function defined below.
\begin{definition}[Comparison]
Let $u\in H^1(\Omega)$ be a local minimizer of the functional \eqref{intro0} and $B_{2R}\subset \subset \Omega$. We shall denote by $v$ the solution of the following problem
\begin{equation}\label{Comparison}
v:=\arg \min
\left\{
\int_{B_{R}}\tilde{F}(x,\nabla w)\,dx: \quad w\in u+H^1_0(B_R)
\right\},
\end{equation}
where $\tilde{F}(x,z):=F(x,u(x),z)$ satisfies the structure condition 
$\eqref{structure1}$ and the growth conditions $\eqref{ellipticity1}$ and $\eqref{ellipticity2}$.
\end{definition}
For the comparison function $v$ defined in \eqref{Comparison} we can state the following reverse H\"{o}lder inequality up to the boundary of $B_R$.

\begin{lemma}
Let $u\in H^1(\Omega)$ be a local minimizer of the functional \eqref{intro0}, $v$ the comparison function defined above and $B_{2R}\subset \subset \Omega$.
Let us consider the following extension of $v$:
\begin{equation*}
V(x) := \left\{
\begin{array}
{ll}v(x) & \mbox{ for } x\in B_R \\
u(x) & \mbox { for } x\in \Omega\setminus B_R.
\end{array}
\right.
\end{equation*}
Denote by $B_{\rho}(x_0)$ a generic ball centered in $x_0\in B_R$ with $\rho<\frac R2$. Then:
\begin{equation}\label{revh2}
\fint_{B_{\rho}(x_0)}|\nabla V|^2\,dx\leq C\Bigl( \fint_{B_{2\rho}(x_0)}(|\nabla V|^{2m}+ |\nabla u|^{2m}+1)\,dx \Bigr)^\frac 12,
\end{equation}
where $m=\frac{n}{n+2}$ and $C=C(n,\nu,N,L)$.
\end{lemma}
\begin{proof}
Let $x_0\in B_R$ and $\rho\leq s<t\leq r <2\rho < R$, where $r=\frac 32 \rho$; then, the following alternative holds,
\begin{equation*}
 \left\{
\begin{array}
{ll}i)& B_{r}(x_0)\subset \subset B_R \\
ii)& \overline{B}_{r}(x_0) \cap (\Omega \setminus B_R)\neq \emptyset.
\end{array}
\right.
\end{equation*}
In the case i) we can proceed exactly as in Lemma 3 to get the desired estimate. Let us then consider the case ii) which is slightly different. Choose $\eta \in C^{\infty}_0(B_{t}(x_0))$, such that $0\leq \eta \leq 1$, $\eta \equiv 1$ in $B_{s}$ and $|\nabla \eta|\leq 2/(t-s)$. Now we can use the following function $\varphi :=\eta(V-u)$ to test the minimality of $v$ with the aim of estimating the difference $\int_{B_s}|\nabla(V-u)|$.  Using the growth conditions $\eqref{ellipticity1}$ and $\eqref{ellipticity2}$ and the minimality of $v$ we obtain
\begin{eqnarray}\label{L31}
&&\frac{\nu}{2}\int_{B_{t}}\Bigl(|\nabla \varphi |^2- \frac{2L^2}{\nu^2}\Bigr) \,dx \leq  \int_{B_{t}\cap B_R}\tilde{F}(x,\nabla \varphi)\,dx=\\
\nonumber& =& \int_{B_{t}\cap B_R}\tilde{F}(x,\nabla v)\,dx+
\int_{B_{t}\cap B_R}\Bigl(\tilde{F}(x,\nabla \varphi)-\tilde{F}(x,\nabla v)\Bigr)\,dx \\
\nonumber &\leq& \int_{B_{t}\cap B_R}\tilde{F}(x,\nabla (v-\varphi))\,dx+ \int_{B_{t}\cap B_R}\sum_{i,j=1}^{n}a_{ij}(x)\D_i(\varphi-v)\D_j(\varphi-v)\,dx \\
\nonumber &-& 2\int_{B_{t}\cap B_R}\sum_{i,j=1}^{n}a_{ij}(x)\D_i\varphi\D_j(\varphi-v)+
\int_{B_{t}\cap B_R}\sum_{i=1}^{n}a_{i}(x)(\D_i\varphi-\D_iv)\,dx\\
\nonumber&\leq&
\frac{\nu}{4}\int_{B_{t}}|\D \varphi|^2\,dx+
(3N+\frac{8N^2}{\nu^2})\int_{B_{t}}|\nabla (V-\varphi)|^2\,dx + \int_{B_{t}}L^2\,dx 
\end{eqnarray}
Where we used Young's inequality in the last estimate. We can summarize the previous estimate as follows
\begin{equation}\label{F1}
\frac{\nu}{4}\int_{B_{t}}|\nabla \varphi |^2 \,dx \leq C(\nu,N,L) \int_{B_{t}}\Bigl (1+
|\nabla(V-\varphi)|^2\Bigr )\,dx .
\end{equation}
Now we observe that $|\nabla(V-\varphi)|\leq|\nabla u|+(1-\eta)|\nabla(V-u)|+\frac{1}{(t-s)}|V-u|$; then by \eqref{F1} we deduce
\begin{eqnarray}\label{F2}
&&\int_{B_{s}}|\nabla (V-u) |^2 \,dx \leq C(\nu,N,L) \int_{B_{t}\setminus B_{s}}|\nabla (V-u) |^2 \,dx\\
\nonumber&&+\frac{C(\nu,N,L)}{(t-s)^2}
\int_{B_{t}}|(V-u)|^{2}\,dx + C(\nu,N,L)\int_{B_{t}}\bigl(1+|\nabla u|^2 \bigr)\,dx.
\end{eqnarray}
Now we use the ``hole filling'' technique adding $C(\nu,N,L)\int_{B_{s}}|\nabla (V-u) |^2 \,dx$ on both sides of \eqref{F2} to get
\begin{equation*}
\begin{split}
&\int_{B_{s}}|\nabla (V-u) |^2 \,dx \leq \theta \int_{B_{t}}|\nabla (V-u) |^2 \,dx\\
&+ C(\nu,N,L)\biggl[\frac{1}{(t-s)^2} \int_{B_{t}}|(V-u)|^{2}\,dx + \int_{B_{t}}\big(1+|\nabla u|^2\big) \,dx\biggr],
\end{split}
\end{equation*}
where $\theta =C(\nu,N,L)/(C(\nu,N,L)+1) $.
Using Lemma \ref{Lemma iterativo 1} we obtain
$$
\int_{B_{\rho}}|\nabla (V-u) |^2\,dx \leq \frac{C(\nu,N,L)}{(r-\rho)^2}\int_{B_{r}}|(V-u)|^{2}\,dx + C(\nu,N,L)\int_{B_{r}}\big(1+|\nabla u|^2\big)\,dx.
$$
Therefore, having chosen $r=\frac 32 \rho$ and by condition ii), we have
$$
|{B}_{2\rho}(x_0) \setminus B_R)|\geq C |{B}_{\rho}(x_0)|,
$$
for some universal constant $C=C(n)$. We can now use Sobolev-Poincar\'{e}'s inequality for functions vanishing on a set of positive measure (see \cite[Inequality (3.29)]{Giu}) to deduce
\begin{equation*}\label{RHV}
\fint_{B_{\rho}(x_0)}\!\!\!\!\!|\nabla (V-u)|^2\,dx\leq C(\nu,N,L)\Biggl[\biggl( \fint_{B_{2\rho}(x_0)}\!\!\!\!\!(|\nabla (V-u)|^{2m}\,dx \biggr)^\frac 1m + \fint_{B_{2\rho}}\!\!\big(1+|\nabla u|^2 \big)\,dx\Biggr]
\end{equation*}
Finally we can apply reverse H\"older inequality $\eqref{ReverseH}$ for $u$ in the last estimate to get $\eqref{revh2}$.
\end{proof}
Reasoning in a similar way as above, the higher integrability for $v$ can be obtained by means of the Gehring's Lemma (see \cite[Proposition 6.1]{Giu}).
\begin{lemma}
\label{Lemma maggiore sommabilità funzione confronto}
Let $u\in H^1(\Omega)$ be a local minimizer of the functional \eqref{intro0}, $v\in H^{1}(B_R(x_0))$ the comparison function defined in $\eqref{Comparison}$. Denoting by $s>1$ the same exponent given in Lemma \ref{Lemma maggiore sommabilità}, it holds
\begin{equation*}
\fint_{B_{R}(x_0)}|\nabla v|^{2s}\,dx\leq C_3\biggl( \fint_{B_{2R}(x_0)}\big(1+|\nabla u|^{2}\big)\,dx\biggr)^s,
\end{equation*}
where $C_3$ depend only on $\nu,N,L,n$.
\end{lemma}

\subsection{Decay estimates for elastic minima}
In this section we prove a decay estimate for elastic minima that will be crucial for the proof strategy. 
As a matter of fact we show that if $(E,u)$ is a $\Lambda$ minimizer of the functional $\eqref{intro0}$ and $x_0$ is a point in $\Omega$, where either the density of E is close to $0$ or $1$, or the set $E$ is asymptotically close to a hyperplane, then for sufficiently small $\rho$ we have
$$
\int_{B_{\rho}(x_0)}|\nabla u_E|^2\,dx\leq C\rho^{n-\delta},
$$
for any $\delta>0$. 
A preliminary result we want to mention, that we will use later, provides an upper bound for $\mathcal{F}$. It is rather standard and is related to the threshold H\"older exponent $\frac 12$ of the function $u$, when $(E,u)$ is either a solution of the constrained problem \eqref{P_c} or a solution of the penalized problem \eqref{P} defined in section 1. For the proof we address the reader too \cite[Lemma 2.3]{LK} and \cite{FJ}. A detailed proof in the case of costrained problems and for functionals satisfying general $p$-polinomial growth is contained in \cite{CFP}.
\begin{theorem}
\label{Energy upper bound}
Let $(E,u)$ be a $\Lambda$-minimizers of the functional ${\mathcal F}$ defined in \eqref{intro0}. For every open set $U\subset \subset \Omega$ there exists a constant $C_1$, depending on $U$ and $\norm{\D u}_{L^2(\Omega)}$, such that for every $B_r(x_0)\subset U$ it holds
$$
{\mathcal F}(E,u;B_r(x_0))\leq C_1r^{n-1}.
$$
\end{theorem}
As a consequence of the previous theorem we can infer, using Poincaré's inequality and the characterization of Campanato spaces (see for example \cite[Theorem 2.9]{Giu}), that $u\in C^{0,\frac 12}$. We deduce the following remark.
\begin{remark}
\label{Osservazione Holderianita}
Let $(E,u)$ be a $\Lambda$-minimizers of the functional ${\mathcal F}$ defined in \eqref{intro0}. For every open set $U\subset \subset \Omega$ there exists a constant $C$, depending on $U$ and $\norm{\D u}_{L^2(\Omega)}$, such that 
\begin{equation}\label{hoelderu}
\sup_{x,y \in U}\frac{|u(x)-u(y)|}{|x-y|^{\frac 12}}\leq C \norm{\D u}_{L^2(\Omega)}.
\end{equation}
\end{remark}

\begin{Notation} In the sequel $E\subset \Omega$ will denote any given subset of $\Omega$ with finite perimeter. We denote by $u_E$, or simply by $u$ if no confusion arises, any local minimizers of the functional $\mathcal{F}(E,v;\Omega)$.
\begin{itemize} 
\item If $x\in \mathbb{R}^n$ we write $x=(x',x_n)$, where $x'\in \mathbb{R}^{n-1}$ and $x_n\in \mathbb{R}$. \\
Accordingly we denote $\D'=(\partial_{x_1},\dots,\partial_{x_{n-1}})$ the gradient with respect to the first $n-1$ components.
\item We will denote $H=\lbrace x\in\Omega\,:\, x_n>0 \rbrace,$
\end{itemize}
\end{Notation}
In order to prove the main lemma of this section we introduce the following preliminary result.
For reader convenience we give here a sketch of the proof which can be found in \cite{MV}. Actually we state here a weaker version that is suitable for our aim.
\begin{lemma}
\label{Lemma semispazio}
Let $v\in H^1(B_1)$ be a solution of
\begin{equation*}
-\textnormal{div}(A\D u)=\textnormal{div}\,G, \qquad \mbox{ in }\mathcal{D}'(B_1),
\end{equation*}
where
$$
G^+:=\mathbbm{1}_H G\in C^{0,\alpha}(H\cap B_1),\qquad 
G^-:=\mathbbm{1}_{H^c}G\in C^{0,\alpha}(H^c\cap B_1),
$$
for some $\alpha>0$ and $A$ is an elliptic matrix satisfying 
$$\nu|\xi|^2\leq A_{ij}(x)\xi_i\xi_j\leq N |\xi|^2,$$ and 
$$
A^+:=\mathbbm{1}_H A\in C^{0,\alpha}(\overline{H}\cap B_1),\qquad 
A^-:=\mathbbm{1}_{H^c}A\in C^{0,\alpha}(\overline{H^c}\cap B_1),
$$
for some $\nu,N>0$. Let us denote
$$
C_A=\max\{\norm{A^+}_{C^{0,\alpha}},\norm{A^-}_{C^{0,\alpha}}\},\qquad C_G=\max\{\norm{G^+}_{C^{0,\alpha}},\norm{G^-}_{C^{0,\alpha}}\}.
$$
Then $v\in L_{loc}^{2,n}(B_1)$. Moreover, there exist two constants $C=C(n,\nu,N,C_A,C_G)$ and $r_0=r_0(n,\nu,N,\norm{G}_{L^\infty},C_A,C_G)$ such that, for any $r<r_0$ with $B_r(x_0)\subset B_1$,
\begin{equation}\label{MVdecay}
\int_{B_{\rho}(x_0)}|\D v|^2\,dx \leq C\Bigl(\frac{\rho}{r}\Bigr)^{n}\int_{B_{r}(x_0)}|\D v|^2\,dx+C\rho^{n}, \quad \forall\, \rho<\frac{r}{4}.
\end{equation}
\end{lemma}
\begin{proof}
Fix $x_0\in B_1$ and let $r$ be such that $B_r(x_0)\subset B_1$.
Let us denote by $a^+$ and $a^-$ the averages of $A$ in $H\cap B_r(x_0)$ and $H^c\cap B_r(x_0)$ respectively. In an analogous way we define $g^+$ and $g^-$ the averages of $G$. For $x\in B_r(x_0)$ we define
$$
\overline{A}:= a^+\mathbbm{1}_H+a^-\mathbbm{1}_{H^c},\qquad \qquad \overline{G}:= g^+\mathbbm{1}_H+g^-\mathbbm{1}_{H^c}.
$$
Notice that by assumption
\begin{equation}\label{MVhoelder}
|A(x)-\overline{A}(x)|\leq C_A r^{\alpha}\qquad\mbox{and}\qquad |G(x)-\overline{G}(x)|\leq C_G r^{\alpha},
\end{equation}
Let $w$ be the solution of
\begin{equation*}
\begin{cases}
-\textnormal{div}(\overline{A}\D w)=\textnormal{div}\overline{G}\\
w=v \mbox{ on }\partial B_r(x_0).
\end{cases}
\end{equation*}
The last equation can be rewritten as
\begin{equation}
\label{coeff cost}
\begin{cases}
-\textnormal{div}(a^+\D w^+)=0 & \text{in }B_r(x_0)\cap H,\\
-\textnormal{div}(a^-\D w^-)=0 & \text{in }B_r(x_0)\cap H^c,\\
w^+=w^-& \text{on }B_r(x_0)\cap\dd H,\\
a^+\D w^+\cdot e_n-a^-\D w^-\cdot e_n=g^+\cdot e_n-g^-\cdot e_n,&\text{on }B_r(x_0)\cap\dd H\\
\end{cases}
\end{equation}
where $w^+:=w\mathbbm{1}_{B_r(x_0)\cap H}$, $w^-:=w\mathbbm{1}_{B_r(x_0)\cap H^c}$.
Set
\begin{equation*}
\overline{D}_c w:=\sum_{i=1}^n \overline{A}_{in}\D_i w+\overline{G}\cdot e_n
\end{equation*}
We notice that $\overline{D}_c w$ has no jumps on the boundary thanks to the transmission condition in \eqref{coeff cost}. 
This allows to prove that the distributional gradient of $\overline{D}_c w$ coincides with the point-wise one.\\
\textbf{Step 1:} 
{\em Tangential derivatives of $w$}. Let us denote with $\tau$ the general direction tangent to the hyperplane $\partial H$. Since $\overline{A}$ and $\overline{G}$ are both constant along the tangential directions, the classical difference quotient method gives that $\D_{\tau}w\in W^{1,2}_{loc}(B_r(x_0))$ and 
$$\mbox{div}(\overline{A}\D (\D_{\tau}w))=0 \qquad\mbox{ in }  B_r(x_0).$$
Hence, Caccioppoli's inequality holds,
\begin{equation}
\label{Caccioppoli}
\int_{B_{\rho}(x)}|\D(\D_{\tau} w)|^2\,dy\leq \frac{c(n,\nu,N)}{\rho^2}\int_{B_{2\rho}(x)}|\D_{\tau} w-(\D_{\tau} w)_{x,2\rho}|^2\,dy,
\end{equation}
for all balls $B_{2\rho}(x)\subset B_r(x_0)$ and, by De Giorgi's regularity theorem, $\D_{\tau} w$ is H\"older continuous and there exists $\gamma=\gamma(n,\nu,N)>0$ such that if $B_{s}(x)\subset B_r(x_0)$ 
\begin{eqnarray}
\label{Eqn1}
&&\int_{B_{\rho}(x)}|\D_{\tau} w-(\D_{\tau} w)_{x,\rho}|^2\,dy\\
\nonumber&&\qquad \qquad\qquad\leq c(n,\nu,N)\bigg( \frac{\rho}{s} \bigg)^{n+2\gamma}\int_{B_{s}(x)}|\D_{\tau} w-(\D_{\tau} w)_{x,s}|^2\,dy,
\end{eqnarray}
for any $\rho\in \big(0,\frac{s}{2}\big)$ and
\begin{equation}
\label{Eqn2}
\max_{B_{\frac{\rho}{2}}(x)}|\D_{\tau} w|^2\leq \frac{c(n,\nu,N)}{\rho^n}\int_{B_{\rho}(x)}|\D_{\tau} w|^2\,dy.
\end{equation}
\textbf{Step 2:} {\em Regularity of $\overline{D}_c w$}. First of all observe that $\D_{\tau}(\overline{D}_c w)=\overline{D}_c(\D_{\tau} w)-\overline{G}\cdot e_n$. This imply by Step 1 that the tangential derivatives of $\overline{D}_c w$ belong to $L^2_{loc}(B_r(x_0))$. Furthermore we can estimate directly by definition of $\overline{D}_c w$
\begin{equation*}
|\D_n(\overline{D}_c w)|\leq c(n,N)|\D \D_{\tau} w|,
\end{equation*}
which implies again by Step1
\begin{equation*}
|\D\overline{D}_c w|\leq c(n,N)|\D \D_{\tau} w|.
\end{equation*}
We can conclude that $\overline{D}_c w\in W^{1,2}_{loc}(B_r(x_0))$. Using Poincaré's inequality and \eqref{Caccioppoli}, we have 
\begin{equation*}
\begin{split}
& \int_{B_\rho(x)}|\overline{D}_c w-(\overline{D}_c w)_{x,\rho}|^2\,dy\leq c(n)\rho^2\int_{B_\rho(x)}|\D(\overline{D}_c w)|^2\,dy\\
& \leq c(n,N)\rho^2\int_{B_\rho(x)}|\D(\D_{\tau}w)|^2\,dy\\
& \leq c(n,\nu,N)\int_{B_{2\rho}(x)}|\D_{\tau}w-(\D_{\tau}w)_{x,2\rho}|^2\,dy,
\end{split}
\end{equation*}
for any $B_{2\rho}(x)\subset B_r(x_0)$. By \eqref{Eqn1} and \eqref{Eqn2} we infer
\begin{equation*}
\begin{split}
&\int_{B_\rho(x)}|\overline{D}_c w-(\overline{D}_c w)_{x,\rho}|^2\,dy\\
& \leq c(n,\nu,N)\bigg( \frac{\rho}{r} \bigg)^{n+2\gamma}\int_{B_{\frac{r}{2}}(x)}|\D_{\tau} w-(\D_{\tau} w)_{x,\frac r2}|^2\,dy\\
&\leq c(n,\nu,N)\bigg( \frac{\rho}{r} \bigg)^{n+2\gamma}\int_{B_r(x_0)}|\D_{\tau} w|^2\,dy,
\end{split}
\end{equation*}
for any $x\in B_{\frac{r}{4}}(x_0)$, $\rho\leq \frac{r}{4}$. Hence by Lemma 4.2 in \cite{MV} (see also \cite[Lemma 7.51]{AFP}), $\overline{D}_c w$ is H\"older continuous and
\begin{equation}
\label{Eqn3}
\begin{split}
\max_{B_{\frac r4}(x_0)}|\overline{D}_c w|^2
& \leq c(n,\nu,N)\int_{B_r(x_0)}|\D_{\tau} w|^2\,dy+\bigg| \fint_{B_{\frac{r}{4}}(x_0)}\overline{D}_c w(y)\,dy \bigg|^2\\
& \leq \frac{c(n,\nu,N)}{r^n}\int_{B_r(x_0)}|\D w|^2\,dy+2\norm{G}^2_{L^\infty}.
\end{split}
\end{equation}
\textbf{Step 3:} {\em Comparison between $v$ and $w$.} Subtracting the equation for $w$ from the equation for $v$ we get
\begin{eqnarray}
&&\int_{B_r(x_0)}\overline{A}_{ij}(x)\big(\D_i v-\D_i w\big)\D_j \varphi\,dx\\
\nonumber && =\int_{B_r(x_0)}\bigl(\overline{A}_{ij}(x)-A_{ij}(x)\bigr)\D_i v \D_j \varphi\,dx +\int_{B_r(x_0)}\bigl(\overline{G}_i-G_i\bigr)\D_i\varphi\,dx
\end{eqnarray}
for any $\varphi \in W^{1,2}_0(B_r(x_0))$. Choosing $\varphi =v-w$ in the previous equation 
and using assumption $\eqref{MVhoelder}$ we have
\begin{equation}
\label{Eqn6}
\int_{B_r(x_0)}|\D v-\D w|^2\,dx\leq C_Ar^{\alpha}\int_{B_r(x_0)}|\D v|^2\,dy+C_G r^{n+\alpha},
\end{equation}
Finally we can estimate
\begin{equation*}
\begin{split}
&\int_{B_\rho(x_0)}|\D v|^2\,dy\leq 2\int_{B_\rho(x_0)}|\D w|^2\,dy+2\int_{B_\rho(x_0)}|\D v-\D w|^2\,dy\\
& \leq 2\omega_n\rho^n \sup_{B_{\frac{r}{4}}}|\D w|^2+2\int_{B_\rho(x_0)}|\D v-\D w|^2\,dy,
\end{split}
\end{equation*}
for any $\rho\leq\frac{r}{4}$ and observing that 
\begin{equation*}
\begin{split}
\sup_{B_{\frac{r}{4}}(x_0)}|\D w|^2
&=\sup_{B_{\frac{r}{4}}(x_0)}|\D_{\tau} w|^2+\sup_{B_{\frac{r}{4}}(x_0)}|\D_n w|^2\\
& \leq c(n,\nu,N)\sup_{B_{\frac{r}{4}}(x_0)}|\D_{\tau} w|^2+c(\nu)\sup_{B_{\frac{r}{4}}(x_0)}|\overline{D}_c w|^2+c(\nu,\norm{G}_{\infty}),
\end{split}
\end{equation*}
by \eqref{Eqn2}, \eqref{Eqn3}, minimality of $w$ and Young's inequality we gain
\begin{equation*}
\begin{split}
& \int_{B_\rho(x_0)}|\D v|^2\,dy\\
& \leq c(n,\nu,N)\bigg( \frac{\rho}{r}\bigg)^n\int_{B_r(x_0)}|\D w|^2\,dy+c(n,\nu,\norm{G}_{\infty},C_A,C_G)\bigg[r^{\alpha}\int_{B_r(x_0)}|\D v|^2\,dy+r^n\bigg]\\
& \leq C(n,\nu,N,\norm{G}_{\infty},C_A,C_G)\Bigg\lbrace\Bigg[ \bigg( \frac{\rho}{r}\bigg)^n+r^{\alpha} \Bigg]\int_{B_r(x_0)}|\D v|^2\,dy+r^n\Bigg\rbrace
\end{split}
\end{equation*}
which leads to our aim if we apply Lemma \ref{Lemma iterativo 2}.
\end{proof}
The next lemma is inspired by \cite[Proposition 2.4]{FJ} and is the main result of this section.
\begin{lemma}
\label{Lemma decadimento 1}
Let $(E,u)$ be a $\Lambda$-minimizers of the functional ${\mathcal F}$ defined in \eqref{intro0}. There exists $0<\tau_0<1$ such that the following statement is true: for all $\tau \in (0,\tau_0)$ there exists $\varepsilon_0=\varepsilon_0(\tau)>0$ such that if $B_r(x_0)\subset \subset \Omega$ with $r^{\frac 1{2n}}<\tau$ and one of the following conditions holds:
\begin{itemize}
\item[\emph{(i)}]$ |E\cap B_r(x_0)|<\varepsilon_0 |B_r|$,
\item[\emph{(ii)}]$ |B_r(x_0)\setminus E|<\varepsilon_0 |B_r|$,
\item[\emph{(iii)}] There exists a halfspace $H$ such that $\frac{\left|(E\Delta H)\cap B_r(x_0)\right|}{|B_r|}<\varepsilon$,
\end{itemize}
then
$$
\int_{B_{\tau r}(x_0)}|\nabla u|^2\,dx\leq C_0 \tau ^{n} \int_{B_r(x_0)}|\nabla u|^2\,dx+C_0r^{n},
$$
for some constant $C_0$ depending only on $n,\nu,N,L,L_D,\norm{\D u}_{L^2(\Omega)}$.
\end{lemma}
\begin{proof}
Let us fix $B_r(x_0)\subset \subset \Omega$ and $0<\tau<1$. Without loss of generality, we may assume that $\tau < 1/4$ and $x_0=0$. We start proving \emph{(i)}, being the proof of \emph{(ii)} similar.
Let us define
\begin{equation*}
A^0_{ij}:=a_{ij}(x_0,u_{r/2}(x_0)),\quad B^0_i:=a_i(x_0,u_{r/2}(x_0)),\quad  f^0:=a(x_0,u_{r/2}(x_0)),
\end{equation*}
\begin{equation*}
F_0(\xi):=(A^0 \xi;\xi)+(B^0;\xi)+f^0.
\end{equation*}
Let us denote by $v$ the solution of the following problem
$$
\min
\left\{\mathcal{F}_0(w;B_{r/2})
: \; w=u \; \mbox{ on }
\partial B_{r/2}
\right\},
$$
where
\begin{equation*}
\mathcal{F}_0(w;B_{r/2}):=\int_{B_{r/2}}F_0(\nabla w)\,dx
\end{equation*}
We use now the following identity 
\begin{equation*}
(A^0 \xi;\xi)-(A^0 \eta;\eta)=(A^0 (\xi-\eta);(\xi-\eta))+2(A^0 \eta;\xi-\eta),
\end{equation*}
in order to deduce that
\begin{eqnarray}\label{EF0}
\nonumber &&\mathcal{F}_0(u)-\mathcal{F}_0(v)\\
\nonumber &&=\int_{B_{r/2}}\bigl[(A^0 \D u;\D u)-(A^0 \D v;\D v)\bigr]\,dx+\int_{B_{r/2}}(B^0;\D u-\D v)\,dx\\
&&=\int_{B_{r/2}}(A^0 (\D u-\D v);(\D u-\D v))\,dx\\
\nonumber &&+2\int_{B_{r/2}}(A^0 \D v;\D u-\D v)\,dx + 
\int_{B_{r/2}}(B^0;\D u-\D v)\,dx
\end{eqnarray}
By the Euler-Lagrange equation for $v$ we deduce that the sum of the last two integrals in the previous identity is zero being also $u=v$ on $\partial B_{r/2}$. Therefore, using the ellipticity assumption of $A^0$ we finally achieve that
\begin{equation}\label{EllF0}
\nu\int_{B_{r/2}}|\D u - \D v|^2\,dx\leq\mathcal{F}_0(u)-\mathcal{F}_0(v).
\end{equation}
We prove now that $u$ is an $\omega$-minimizer of $\mathcal{F}_0$. We start writing
\begin{eqnarray}
\label{omegamin}
\nonumber &&\mathcal{F}_0(u)
 =\mathcal{F}(E,u)+[\mathcal{F}_0(u)-\mathcal{F}(E,u)]\\
&&\leq \mathcal{F}(E,v)+[\mathcal{F}_0(u)-\mathcal{F}(E,u)]\\
\nonumber && = \mathcal{F}_0(v)+[\mathcal{F}_0(u)-\mathcal{F}(E,u)]+[\mathcal{F}(E,v)-\mathcal{F}_0(v)].
\end{eqnarray}
\emph{Estimate of $\mathcal{F}_0(u)-\mathcal{F}(E,u)$}. We use \eqref{ellipticity2}, \eqref{hoelderu} and \eqref{hoelderu} to infer
\begin{equation}
\label{eqq6}
\begin{split}
&\mathcal{F}_0(u)-\mathcal{F}(E,u)=\int_{B_{r/2}}
\bigl(a_{ij}(x_0,u_{r/2}(x_0))-a_{ij}(x,u(x))\bigr)\D_iu\D_ju\,dx\\
&+\int_{B_{r/2}}\bigl(a_i(x_0,u_{r/2}(x_0))-a_i(x,u(x))\bigr)\D_iu\,dx \\
&+\int_{B_{r/2}}\bigl(a(x_0,u_{r/2}(x_0))-a(x,u(x))\bigr)\,dx -\int_{B_{r/2}\cap E}G(x,u,\D u)\,dx\\
&\leq 2L_D\norm{\D u}_{L^2(\Omega)}\biggl(r^{\frac 12}\int_{B_{r/2}}|\D u|^2\,dx+r^{n+\frac 12}\biggr)+C(N,L)\int_{B_{r/2}\cap E}|\D u|^2\,dx +2Lr^n
\end{split}
\end{equation}
where we denoted $L_D$ the greatest lipschitz constant of the data 
$a_{ij},b_{ij}, a_i,b_i,a,b$ defined in \eqref{Hoelderianity}.
Now we use Lemma \ref{Lemma maggiore sommabilità} to estimate
\begin{eqnarray}\label{improvement}
\nonumber&&\int_{B_{r/2}\cap E}|\D u|^2\,dx\leq |E\cap B_r|^{1-1/s}|B_r|^{1/s}\biggl(\fint_{B_{r/2}}|\D u|^{2s}\biggr)^{1/s}\\
&&\leq C_2^{1/s}\biggl(\frac{|E\cap B_r|}{|B_r|}\biggr)^{1-1/s}\int_{B_r}\big(1+|\D u|^2\big)\,dx.
\end{eqnarray}
Merging the last estimate in $\eqref{eqq6}$ we deduce
\begin{eqnarray}\label{E0}
\nonumber &&\mathcal{F}_0(u)-\mathcal{F}(E,u)
\leq\Bigl(L_D\norm{\D u}_{L^2(\Omega)}
+C(N,L)C_2^{1/s}\Bigr)
\Bigl(r^{\frac 12}+\varepsilon_0^{1-1/s}\Bigr)
\int_{B_r}|\D u|^2\,dx\\
&&+ \Bigl(C_{2}^{1/s}+2L+L_D\norm{\D u}_{L^2(\Omega)}\Bigr)r^n
\end{eqnarray}
{\em Estimate of $\mathcal{F}(E,v)-\mathcal{F}_0(v)$}.
\begin{eqnarray}\label{E1}
\nonumber &&\mathcal{F}(E,v)-\mathcal{F}_0(v)=\int_{B_{r/2}}
\bigl(a_{ij}(x,v(x))-a_{ij}(x_0,u_{r/2}(x_0))\bigr)\D_iv\D_jv\,dx\\
&&+\int_{B_{r/2}}\bigl(a_i(x,v(x))-a_i(x_0,u_{r/2}(x_0))\bigr)\D_iv\,dx \\
\nonumber && +\int_{B_{r/2}}\bigl(a(x,v(x))-a(x_0,u_{r/2}(x_0))\bigr)\,dx
+\int_{B_{r/2}\cap E}G(x,v,\D v)\,dx.
\end{eqnarray}
If we choose now $z\in\partial B_{r/2}$, recalling that $u(z)=v(z)$ we deduce
\begin{eqnarray*}
&&\bigl|a_{ij}(x,v(x))-a_{ij}(x_0,u_{r/2}(x_0))\bigr|\\
&&=\bigl|a_{ij}(x,v(x))-
a_{ij}(x,v(z))+a_{ij}(x,u(z))-a_{ij}(x_0,u_{r/2}(x_0))\bigr|\\
&&\leq L_D\big(|v(x)-v(z)|+r^{\frac 12}\norm{\nabla u}_{L^2(\Omega)}+r\big)\\
&&\leq L_D\big(\text{osc}(u,\partial B_{r/2})+C(n,\nu,N,L)r+r^{\frac 12}\norm{\nabla u}_{L^2(\Omega)}+r\big)\\
&&\leq C(n,\nu,N,L,L_D,\norm{\D u}_{L^2(\Omega)})r^{\frac 12}
\end{eqnarray*}
where we used the fact that $\text{osc}(v,B_{r/2})\leq \text{osc}(u,\partial B_{r/2})+C(n,\nu,N,L)r$,  (see \cite[Lemma 8.4]{Giu}). Analogously we can estimate the other difference in $\eqref{E1}$, deducing
\begin{eqnarray*}
&&\mathcal{F}(E,v)-\mathcal{F}_0(v)\leq C(n,\nu,N,L,L_D,\norm{\D u}_{L^2(\Omega)}) r^{\frac 12}\biggl(\int_{B_{r/2}}|\D v|^2\,dx+ r^n\biggr)\\
&&+C(N,L)\biggl(\int_{B_{r/2}\cap E}|\D v|^2\,dx +r^n\biggr),
\end{eqnarray*}
Reasoning in a similar way as in \eqref{improvement}, we can apply the higher integrability for $v$ given by Lemma \ref{Lemma maggiore sommabilità funzione confronto} and infer
$$
\int_{B_{r/2}\cap E}|\D v|^2\,dx \leq C(n,\nu,N,L)\varepsilon_0^{1-1/s}\biggl(\int_{B_{r}}|\D u|^2\,dx+r^n\biggr).
$$
Therefore we obtain
\begin{eqnarray}\label{E2}
&&\mathcal{F}(E,v)-\mathcal{F}_0(v)\\
\nonumber &&\leq C(n,\nu,N,L,L_D,\norm{\D u}_{L^2(\Omega)})\biggl[\Bigl(r^{\frac 12}+\varepsilon_0^{1-1/s}\Bigr)\int_{B_r}|\D u|^2\,dx + r^n\biggr].
\end{eqnarray}
Finally, collecting $\eqref{EllF0}$, $\eqref{omegamin}$, $\eqref{E0}$ and $\eqref{E2}$, if we choose $\varepsilon_0$ such that $\varepsilon_0^{1-\frac{1}{s}}=\tau^n$ we conclude that
\begin{equation}\label{FEstimate}
\int_{B_{r/2}}|\D u - \D v|^2\,dx\leq C \tau^n\int_{B_r}|\D u|^2\,dx+ Cr^n,
\end{equation}
for some constant $C=C(n,\nu,N,L,L_D,\norm{\D u}_{L^2(\Omega)})$. On the other hand $v$ is the solution of a uniform elliptic equation with constant coefficients, so we have
\begin{equation}\label{FEstimate1}
\int_{B_{\tau r}}|\D v|^2\,dx\leq C_1(n,\nu,N)\tau^n \int_{B_{r/2}}|\D v|^2\,dx\leq C_2(n,\nu,N)\tau^n \int_{B_{r/2}}|\D u|^2\,dx.
\end{equation}
Hence we may estimate, using \eqref{FEstimate} and \eqref{FEstimate1}
\begin{eqnarray}
\int_{B_{\tau r}}|\D u|^2\,dx&\leq& 2 \int_{B_{\tau r}}|\D v -\D u|^2\,dx +2 \int_{B_{\tau r}}|\D v|^2\,dx\\
&\leq& C_0 \tau^n\int_{B_r}|\D u|^2\,dx+ C_0r^n,
\end{eqnarray}
for some constant $C_0=C_0(n,\nu,N,L,L_D,\norm{\D u}_{L^2(\Omega)})$.\\
We are left with the case \emph{(iii)}. Let $H$ be the half space from the assumption and let us denote accordingly
\begin{eqnarray*}
&& A^0_{ij}(x):=a_{ij}(x,u(x))+\mathbbm{1}_{H}b_{ij}(x,u(x)) ,\\
&& B^0_{ij}(x):=a_i(x,u(x))+\mathbbm{1}_{H}b_i(x,u(x)),\\
&&f^0(x):= a(x,u(x))+\mathbbm{1}_{H} b(x,u(x)),\\
&&F_0(x,\xi):=(A^0(x) \xi;\xi)+(B^0(x);\xi)+f^0(x).
\end{eqnarray*}
Let us denote by $v_H$ the solution of the following problem
$$
\min
\left\{\mathcal{F}_0(w;B_{r/2})
: \; w=u \; \mbox{ on }
\partial B_{r/2}
\right\},
$$
where
\begin{equation*}
\mathcal{F}_0(w;B_{r/2}):=\int_{B_{r/2}}F_0(x,\nabla w)\,dx.
\end{equation*}
Let us point out that $v_H$ solves the Euler-Lagrange equation
\begin{equation}\label{MV}
-2\textnormal{div}(A^0\D v_H)=\textnormal{div}\,B^0 \quad \text{in }\mathcal{D}'(B_{r/2}).
\end{equation}
Therefore we are in position to apply Lemma \ref{Lemma semispazio} to the function $v_H$. As a matter of fact, from the H\"older continuity of $u(x)$ (see Remark \ref{Osservazione Holderianita}) we deduce that the restrictions of $A^0$ and $B^0$ onto $H\cap B_r$ and $B_r\setminus H$ respectively, are H\"older continuous. We can conclude using also $\eqref{hoelderu}$ that there exist two constants $C=C(n,\nu,N,L,L_D,\norm{\D u}_{L^2(\Omega)})$ and $\tau_0=\tau_0(n,\nu,N,L,L_D,\norm{\D u}_{L^2(\Omega)})$ such that for $\tau<\tau_0$,
\begin{equation}\label{Decayvh}
\int_{B_{\tau r}}
|\D v_H|^2\,dx\leq 
C\tau^n\int_{B_{r/2}}
|\D v_H|^2\,dx + Cr^n.
\end{equation}
In addition, using the ellipticity condition of $A^0$ we can argue as in $\eqref{EF0}$ to deduce using also the fact that $v_H$ satisfies $\eqref{MV}$,
\begin{equation}\label{EllF1}
\nu\int_{B_{r/2}}|\D u - \D v_H|^2\,dx\leq\mathcal{F}_0(u)-\mathcal{F}_0(v_H).
\end{equation}
One more time we can prove that $u$ is an $\omega$- minimizer of $\mathcal{F}_0$. We start as above writing
\begin{eqnarray}
\label{omegaminH}
\nonumber &&\mathcal{F}_0(u)
 =\mathcal{F}(E,u)+[\mathcal{F}_0(u)-\mathcal{F}(E,u)]\\
&&\leq \mathcal{F}(E,v_H)+[\mathcal{F}_0(u)-\mathcal{F}(E,u)]\\
\nonumber && = \mathcal{F}_0(v_H)+[\mathcal{F}_0(u)-\mathcal{F}(E,u)]+[\mathcal{F}(E,v_H)-\mathcal{F}_0(v_H)].
\end{eqnarray}
We can estimate the differences $\mathcal{F}_0(u)-\mathcal{F}(E,u)$ and $\mathcal{F}(E,v_H)-\mathcal{F}_0(v_H)$ exactly as before using this time the higher integrability given in Lemma \ref{Lemma maggiore sommabilità funzione confronto}.
We conclude that
\begin{equation}\label{FEHstimate}
\int_{B_{r/2}}|\D u - \D v_H|^2\,dx\leq C \tau^n\int_{B_r}|\D u|^2\,dx+ Cr^n,
\end{equation}
for some constant $C=C(n,\nu,N,L,L_D,\norm{\D u}_{L^2(\Omega)})$. From the last estimate we can conclude the proof as before using \eqref{Decayvh} and \eqref{EllF1}.
\end{proof}
\bigskip
\section{Energy density estimates}
This section is devoted to prove a lower bound esimate for the functional $ {\mathcal F}(E,u;B_r (x_0))$. Lemma \ref{Lemma decadimento 1} is the main tool to achieve such result. We shall prove that the energy $\mathcal F$ decays ``fast'' if the perimeter of $E$ is ``small''. 

\begin{lemma}
\label{Lemma decadimento 2} Let $(u,E)$ be a $\Lambda$-minimizer of the functional ${\mathcal F}$ defined in \eqref{intro0}. For every $\tau\in (0,1)$ there exists $\varepsilon=\varepsilon_1(\tau)>0$ such that, if $B_r(x_0)\subset \Omega$ and $P(E;B_r(x_0))<\varepsilon_1 r^{n-1}$, then
\begin{equation}\label{D}
\mathcal F(E,u;B_{\tau r}(x_0))\leq C_1 \tau^n\bigl(\mathcal F(E,u;B_r(x_0))+r^n\bigr),
\end{equation}
for some constant $C_1=C_1(n,\nu,N,L,L_D,\norm{\D u}_{L^2(\Omega)})$ independent of $\tau$ and $r$.
\end{lemma}
\begin{proof}
Without loss of generality we may assume that $\tau<\frac 12$. We can also assume that $x_0=0,\ r=1$ by introducing the following rescaling, $E_r=\frac{E-x_0}{r}$, $u_r=r^{-\frac 12}{u(x_0+ry)}$ and replacing $\Lambda$ with $\Lambda r$. Thus, we have that $(E_r,u_r)$ is a $\Lambda r$-minimizer of $\mathcal F$ in $\frac{\Omega-x_0}{r}$. For simplicity of notation we can still denote $E_r$ by $E$, $u_r$ by $u$ and then we have to prove that, given $0<\tau<\frac 12$, there exists $\varepsilon_1=\varepsilon_1(\tau)$ such that, if $P(E;B_1)<\varepsilon_1$, then
$$
\mathcal F(E,u;B_{\tau })\leq C_1 \tau^n\Bigl(\mathcal F(E,u;B_1)+r\Bigr).
$$
Note that, since $P(E;B_1)$ is small, holds, by the relative isoperimetric inequality $|B_1\cap E|$ or $|B_1\setminus E|$ is small. Thus Lemma \ref{Lemma decadimento 1} holds. Assuming that $|B_1\setminus E|$ is small and using the relative isoperimetric inequality we can deduce that,
$$
|B_1\setminus E|\leq c(n)P(E;B_1)^{\frac{n}{n-1}}.
$$
If we choose as a representative of $E$ the set of points of density one, we get, by Fubini's theorem,
$$
|B_1\setminus E|\geq \int_{\tau}^{2\tau}\mathcal H^{n-1}(\partial B_\rho\setminus E)\ d\rho.
$$
Combining these inequalities we can choose $\rho\in (\tau, 2\tau)$ such that
\begin{equation}\label{perimeter}
\mathcal{H}^{n-1}(\partial B_\rho\setminus E)\leq \frac{c(n)}{\tau}P(E;B_1)^{\frac{n}{n-1}}\leq \frac{c(n)\varepsilon_1^{\frac{1}{n-1}}}{\tau}P(E;B_1).
\end{equation}
Now we set $F=E\cup B_\rho$ and observe that
$$
P(F;B_1)\leq P(E;B_1\setminus \overline{B}_{\rho})+\mathcal H^{n-1}(\partial B_{\rho}\setminus E).
$$
If we choose $(u,F)$ to test the $\Lambda r$-minimality of $(u,E)$ we get
\begin{eqnarray*}
&& P(E;B_1)+\int_{B_1}\Bigl ( F(x,u,\nabla u)+\mathbbm{1}_{E}G(x,u,\nabla u)\Bigr )\,dx\leq \\
&\leq& P(F;B_1)+ \int_{B_1}\Bigl ( F(x,u,\nabla u)+\mathbbm{1}_{F}G(x,u,\nabla u)\Bigr )\,dx + \Lambda r|F\setminus E|\\
&\leq& P(E;B_1\setminus \overline{B}_{\rho})+\mathcal H^{n-1}(\partial B_{\rho}\setminus E)+ \int_{B_1}\Bigl ( F(x,u,\nabla u)+\mathbbm{1}_{F}G(x,u,\nabla u)\Bigr )\,dx + \Lambda r |B_{\rho}|.
\end{eqnarray*}
Then getting rid of the common terms we obtain
\begin{equation*}
P(E;B_\rho)\leq \mathcal H^{n-1}(\partial B_{\rho}\setminus E)+ \int_{B_{\rho}} G(x,u,\nabla u)\,dx + \Lambda r |B_{\rho}|.
\end{equation*}
Now if we choose $\varepsilon_1$ such that $c(n)\varepsilon_1^{\frac{1}{n-1}}\leq \tau^{n+1}$ we have from \eqref{perimeter}
\begin{equation*}
P(E;B_\rho)\leq \tau^n P(E;B_1)+ \int_{B_{\rho}} G(x,u,\nabla u)\,dx + \Lambda r |B_{\rho}|.
\end{equation*}
Then we choose $\varepsilon_1$ satisfying $c(n)\varepsilon_1^{\frac{n}{n-1}} \leq \varepsilon_0(2\tau)|B_1|$ to obtain using Lemma \ref{Lemma decadimento 1} and growth conditions \eqref{ellipticity1}, \eqref{ellipticity2},
\begin{eqnarray*}
&&\int_{B_{\rho}} G(x,u,\nabla u)\,dx \leq C(N,L)\int_{B_{\rho}} (1+|\nabla u|^2)\,dx\\
&&\leq C(n,\nu,N,L,L_D,\norm{\D u}_{L^2(\Omega)}) \tau^n\int_{B_{1}}
(1+|\nabla u|^2)\,dx\\
&\leq& C(n,\nu,N,L,L_D,\norm{\D u}_{L^2(\Omega)}) \tau^n\int_{B_{1}} G(x,u,\nabla u)\,dx.
\end{eqnarray*}
Finally, we recall that $\rho\in (\tau,2\tau)$ to get
\begin{equation*}
P(E;B_\tau)\leq \tau^n P(E;B_1)+ C(n,\nu,N,L,L_D,\norm{\D u}_{L^2(\Omega)})\tau^n \int_{B_{1}} G(x,u,\nabla u)\,dx + \Lambda r |B_{2\tau}|.
\end{equation*}
From this estimate the result easily follows.
\end{proof}
\begin{theorem}[Density lower bound]
\label{Density lower bound}
Let $(u,E)$ be a $\Lambda$-minimizer of $\mathcal{F}$ and $U\subset \subset \Omega$. There exists a costant $C$ depending only on $n,\nu,N,L,L_D,\norm{\D u}_{L^2(\Omega)}$,
such that for every $x_0\in \partial E$ and $B_r(x_0)\subset U$,
$$
P(E,B_r(x_0))\geq Cr^{n-1}.
$$
Moreover, $\mathcal{H}^{n-1}((\partial E\setminus \partial^*E)\cap \Omega)=0$.
\end{theorem}
\begin{proof}
The proof follows from Lemma \ref{Lemma decadimento 1} and Lemma \ref{Lemma decadimento 2} in a standard way, see \cite[Proposition 4.4]{FJ} or \cite[Theorem 7.21]{AFP}.
\end{proof}


\section{Compactness for sequences of minimizers}
In this section we basically follow the route given in \cite[Part III]{Ma}.
We start proving a standard compactness result.
\begin{lemma}[Compactness]
\label{Lemma compattezza}
Let $(E_h,u_h)$ be a sequence of $\Lambda_h$-minimizers of $\mathcal{F}$ in $\Omega$ such that $\sup_h \mathcal F(E,u_h;\Omega)<+\infty$ and 
$\Lambda_h\rightarrow \Lambda\in \mathbb R^+$. There exists a (not relabelled) subsequence and a $\Lambda$-minimizer  of 
$\mathcal F$, $(u,E)$,  such that for every open set $U\subset \subset \Omega$, it holds
$$
E_h\rightarrow E \mbox { in } L^1(U),\;\; u_h\rightarrow u \mbox { in } H^{1}(U),\;\; P(E_h,U)\rightarrow P(E,U).
$$
In addition,
\begin{eqnarray}
&& \label{boundary1} \mbox{if }x_h\in \partial E_h\cap U \mbox{ and } x_h\rightarrow x \in U, \mbox { then } x\in \partial E \cap U,\\
&& \label{boundary2}\mbox{if }x\in \partial E\cap U \mbox{ there exists } x_h\in \partial E_h\cap U \mbox{ such that } x_h\rightarrow x.
\end{eqnarray}
Finally if we assume also $\nabla u_h \rightharpoonup 0$ weakly in $L^2_{loc}(\Omega,\mathbb R^n)$ and $\Lambda_h \rightarrow 0$, then $E$ is a local minimizer 
of the perimeter, that is
$$
P(E,B_r(x_0))\leq P(F,B_r(x_0)),
$$
for every $F$ such that $F\Delta E \subset\subset B_r(x_0)\subset \Omega.$
\end{lemma}
\begin{proof}
We start observing that, by the boundedness condition on $\mathcal F(E,u_h;\Omega)$, we can assume that, $u_h$ weakly converges to $u$ in $H^{1}(U)$ and strongly in $L^2(U)$, and 
$\mathbbm{1}_{E_h}$ converges to $\mathbbm{1}_{E}$ in $L^1(U)$. Using lower semicontinuity we are going to prove the $\Lambda$-minimality of $(u,E)$. Let us fix 
$B_r(x_0)\subset\subset\Omega$ and assume for simplicity of notation that $x_0=0$. Let $(v,F)$ be a test pair such that $F\Delta E \subset \subset B_r$ 
and supp$(u-v)\subset \subset B_r$. 
We can handle the perimeter term as in \cite{Ma}, eventually passing to a subsequence and using Fubini's theorem, we can choose $\rho<r$ such that, once again, 
$F\Delta E \subset \subset B_{\rho}$ and supp$(u-v)\subset \subset B_{\rho}$, and in addition,
$$
\mathcal{H}^{n-1}(\partial^*F\cap \partial B_{\rho})=\mathcal{H}^{n-1}(\partial^*E_h\cap \partial B_{\rho})=0,
$$
and
\begin{equation}
\label{mismatch}
\lim_{h\rightarrow 0}\mathcal{H}^{n-1}(\partial B_{\rho}\cap(E\Delta E_h))=0.
\end{equation}
Now we choose a cut-off function $\psi\in C_0^1(B_r)$ such that $\psi\equiv 1$ in $B_{\rho}$, and define $v_h=\psi v+(1-\psi)u_h$, 
$F_h:=(F\cap B_{\rho})\cup (E_h\setminus B_{\rho})$ to test the minimality of $(u_h,E_h)$. Thanks to the $\Lambda_h$-minimality of $(u_h,E_h)$ we have
\begin{eqnarray}\label{min}
&& \int_{B_r} \Bigl ( F(x,u_h,\nabla u_h)+\mathbbm{1}_{E_h}G(x,u_h,\nabla u_h)\Bigr )\,dx+ P(E_h,B_r)\leq\\
&&\nonumber \leq \int_{B_r} \Bigl ( F(x,v_h,\nabla v_h)+\mathbbm{1}_{F_h}G(x,v_h,\nabla v_h)\Bigr )\,dx+P(F_h,B_r)
+\Lambda_h|F_h\Delta E_h| \\
&&\nonumber \leq \int_{B_r} \Bigl (F(x,v_h,\nabla v_h)+\mathbbm{1}_{F_h}G(x,v_h,\nabla v_h)\Bigr )\,dx +P(F,B_{\rho})+\Lambda_h|F_h\Delta E_h| \\
&& +P(E_h,B_r\setminus \overline{B}_{\rho})+\varepsilon_h.\nonumber
\end{eqnarray}
The mismatch term $\varepsilon_h=\mathcal{H}^{n-1}(\partial B_{\rho}\cap (F^{(1)}\Delta E_{h}^{(1)}))$ appears because $F$ is not in general a compact variation of $E_h$. Nevertheless we have that  $\varepsilon_h\rightarrow 0$ because of the assumption $\eqref{mismatch}$ (see also \cite[Theorem 21.14]{Ma}).\\
Now we use the convexity of $F$ and $G$ with respect to the $z$ variabile to deduce
\begin{eqnarray*}
&& \int_{B_r} \Bigl ( F(x,v_h,\nabla v_h)+\mathbbm{1}_{F_h}G(x,v_h,\nabla v_h)\Bigr )\,dx\\
&& \leq \int_{B_r} \Bigl ( F(x,v_h,\psi \nabla v+(1-\psi)\nabla u_h)+\mathbbm{1}_{F_h}G(x,v_h,\psi \nabla v+(1-\psi)\nabla u_h)\Bigr )\,dx\\
&& + \int_{B_r}  \left < \nabla_z F(x,v_h,\nabla v_h),\nabla \psi(v-u_h)\right >\,dx \\
&& + \int_{B_r}  \mathbbm{1}_{F_h} \left < \nabla_z G(x,v_h,\nabla v_h),\nabla \psi(v-u_h)\right >\,dx\\
\end{eqnarray*}
where last two terms in the previous estimate tend to zero for $h\rightarrow \infty$. As a matter of fact the term $\nabla \psi(v-u_h)$ strongly converges to zero in $L^2$, being $u=v$ in $B_r\setminus B_{\rho}$ and the first part in the scalar poduct weakly converges in $L^2$.
Then using again the convexity of $F$ and $G$ in the $z$ variable we obtain, for some infinitesimal $\sigma_h$,
\begin{eqnarray}\label{compact}
&&\int_{B_r} \Bigl ( F(x,v_h,\nabla v_h)+\mathbbm{1}_{F_h}G(x,v_h,\nabla v_h)\Bigr)\,dx \\
&& \nonumber \leq \int_{B_r} \psi \Bigl( F(x,v_h,\nabla v  )+\mathbbm{1}_{F_h}G(x,v_h,\nabla v)\Bigr)\,dx\\ 
&& + \int_{B_r} (1-\psi) \Bigl( F(x,v_h,\nabla u_h )+\mathbbm{1}_{F_h}G(x,v_h,\nabla u_h)\Bigr)\,dx+ \sigma_h. \nonumber
\end{eqnarray}
Finally, we connect $\eqref{min}$ and $\eqref{compact}$ and pass to the limit with respect to 
$h$ using the lower semicontinuity on the left hand side. For the right hand side we observe that $\mathbbm{1}_{E_h}\rightarrow \mathbbm{1}_{E}$  and $\mathbbm{1}_{F_h}\rightarrow \mathbbm{1}_{F}$ in $L^1(B_r)$ and use also the equi-integrability of $\left\{\nabla u_h\right\}$ to conclude,
\begin{eqnarray*}
&&\int_{B_r} \psi\Bigl ( F(x,u,\nabla u)+\mathbbm{1}_{E}G(x,u,\nabla u)\Bigr )\,dx+ P(E,B_{\rho})\leq\\
&&\leq \int_{B_r} \psi \Bigl ( F(x,v,\nabla v)+\mathbbm{1}_{F}G(x,v,\nabla v)\Bigr )\,dx+ P(F,B_{\rho})+ \Lambda|F\Delta E|\nonumber.
\end{eqnarray*}
Letting $\psi \downarrow \mathbbm{1}_{B_{\rho}}$ we finally get
\begin{eqnarray}\label{lmini}
&&\int_{B_{\rho}} \Bigl ( F(x,u,\nabla u)+\mathbbm{1}_{E}G(x,u,\nabla u)\Bigr )\,dx+ P(E,B_{\rho})\leq\\
&&\leq \int_{B_{\rho}}  \Bigl ( F(x,v,\nabla v)+\mathbbm{1}_{F}G(x,v,\nabla v)\Bigr )\,dx+ P(F,B_{\rho})+ \Lambda|F\Delta E|\nonumber,
\end{eqnarray}
and this proves the $\Lambda$-minimality of $(u,E)$.
\\To prove the strong convergence of $\nabla u_{h}$ to $\nabla{u}$ in $L^2(B_r)$ we start observing that 
by $\eqref{min}$ and $\eqref{compact}$ applied using $(u,E_h)$ to test the minimality of $(u_h,E_h)$ we get
\begin{eqnarray*}
&&\int_{B_r} \psi\Bigl ( F(x,u_h,\nabla u_h)+\mathbbm{1}_{E_h}G(x,u_h,\nabla u_h)\Bigr )\,dx\\
&&\hspace{2cm}\leq
\int_{B_r} \psi \Bigl( F(x,u,\nabla u  )+\mathbbm{1}_{E_h}G(x,u,\nabla u)\Bigr)\,dx+ \sigma_h
\end{eqnarray*}
Then from equiintegrability of $\left\{\nabla u_h\right\}$  in $L^2(U)$ and recalling that $\mathbbm{1}_{E_h}\rightarrow \mathbbm{1}_{E}$ in $L^1(U)$, we obtain
\begin{eqnarray*}
&&\limsup_{h\rightarrow +\infty}\int_{B_r} \psi \Bigl( F(x,u_h\nabla u_h  )+\mathbbm{1}_{E_h}G(x,u_h,\nabla u_h)\Bigr)\,dx \\
&&\hspace{2cm}\leq \int_{B_r} \psi \Bigl( F(x,u,\nabla u  )+\mathbbm{1}_{E}G(x,u,\nabla u)\Bigr)\,dx.
\end{eqnarray*}
The opposite inequality can be obtained by semicontiuity
then we can deduce,
\begin{eqnarray}\label{norm}
&&\lim_{h\rightarrow +\infty}\int_{B_r} \psi \Bigl( F(x,u_h,\nabla u_h  )+\mathbbm{1}_{E_h}G(x,u_h,\nabla u_h)\Bigr)\,dx \\
\nonumber &&\hspace{3cm}= \int_{B_r} \psi \Bigl( F(x,u,\nabla u  )+\mathbbm{1}_{E}G(x,u,\nabla u)\Bigr)\,dx.
\end{eqnarray}
Now from ellepticity condition $\eqref{ellipticity1}$ we infer, for some $\sigma_h\rightarrow 0$,
\begin{eqnarray}
&&\nu\int_{B_r}\psi|\nabla u_h-\nabla u|^2\,dx \leq 
\int_{B_r}\psi\Bigl(F(x,u_h,\nabla u_h)-F(x,u,\nabla u)\Bigr)\,dx \nonumber \\
&&+\int_{B_r}\psi\mathbbm{1}_{E}\Bigl(G(x,u_h,\nabla u_h)-G(x,u,\nabla u)\Bigr)\,dx  + \sigma_h
\end{eqnarray}
Passing to the limit we obtain
$$
\lim_{h\rightarrow +\infty}\int_{B_r}\psi|\nabla u_h -\nabla u |^2\,dx=0.
$$
Finally testing the minimality of $(u_h,E_h)$ with the pair $(u,E)$ we also get
$$
\lim_{h\rightarrow +\infty} P(E_h,B_{\rho})=P(E,B_{\rho}).
$$
With an usual argument we can deduce $u_h\rightarrow u$ in $W^{1,2}(U)$ and $P(E_h,U)\rightarrow P(E,U)$ for every open set $U\subset \subset \Omega$.
The topological information stated in $\eqref{boundary1}$ and $\eqref{boundary2}$ follows as in \cite[Theorem 21.14]{Ma} because doesn't depend on the presence of the integral bulk part.
\end{proof}

\section{Decay of the excess}
\subsection{Excess and height bound}
Now we introduce the usual quantities involved in regularity theory. Given $x\in \partial E$, a scale $r>0$ and a direction $\nu \in \mathbbm S^{n-1}$ we define the {\it spherical excess}
$$
{\mathbf e}(x,r,\nu):= \frac{1}{r^{n-1}}\int_{\partial E\cap B_r(x)}\frac{|\nu_E(y)-\nu|^2}{2}d\mathcal H^{n-1}(y),
$$
and
$$
{\mathbf e}(x,r):=\min_{\nu \in \mathbbm S^{n-1}}{\mathbf e}(x,r,\nu).
$$
In addition we define the {\it rescaled Dirichlet integral} of $u$
$$
\mathcal{D}(x,r):= \frac{1}{r^{n-1}}\int_{B_r(x)}|\nabla u|^2 dy.
$$
The following height bound lemma is a standard step in the proof of regularity.
\begin{lemma}[Height bound]
Let $(u,E)$ be a $\Lambda$-minimizer of $\mathcal{F}$ in $B_r(x_0)$. There exist two positive constants $C$ and $\varepsilon$, depending on $\norm{\D u}_{L^2(B_r(x_0))}$, such that if $x_0\in \partial E$ and
$$
{\mathbf e}(x,r,\nu)<\varepsilon
$$
for some $\nu\in \mathbbm S^{n-1}$ then
$$
\sup_{y\in \partial E\cap B_{r/2}(x_0)}\frac{|\langle \nu,y-x_0\rangle|}{r}\leq C {\mathbf e}(x,r,\nu)^{\frac{1}{2(n-1)}}
$$
\end{lemma}
\begin{proof}
The proof of this lemma is almost equal to the one of \cite[Theorem 22.8]{Ma}. As a matter of fact it follows from the density lower bound (see Theorem \ref{Density lower bound}), the relative isoperimetric inequality and the compactness result proved above.
\end{proof}
Proceeding as in \cite{Ma} we give the following Lipschitz approximation lemma which is a consequence of the height bound lemma. Its proof follows exactly as in \cite[Theorem 23.7]{Ma}.
\begin{theorem}[Lipschitz approximation]
Let $(u,E)$ be a $\Lambda$-minimizer of $\mathcal{F}$ in $B_r(x_0)$.
There exist two positive constants $C_3$ and $\varepsilon_3$, depending on $\norm{\D u}_{L^2(B_r(x_0))}$, such that
$$
\mbox{ if }x_0\in \partial E\;\; \mbox{ and } \;\;{\mathbf e}(x_0,r,e_n)<\varepsilon _3
$$
then there exists a Lipschitz function $f:\mathbbm R^{n-1}\rightarrow \mathbbm R$ such that
$$
\sup_{x'\in \mathbbm R^{n-1}}\frac{|f(x')|}{r}\leq C_3{\mathbf e}(x_0,r,e_n)^{\frac{1}{2(n-1)}},\hspace{1cm} \norm{\nabla'f}_{L^{\infty}}\leq 1
$$
and
$$
\frac{1}{r^{n-1}}\mathcal{H}^{n-1}((\partial E \Delta \Gamma_f)\cap B_{r/2}(x_0))\leq C_3 {\mathbf e}(x_0,r,e_n),
$$
where $\Gamma_f$ is the graph of $f$. Moreover $f$ is ``almost harmonic'' in the sense that
$$
\frac{1}{r^{n-1}}\int_{B_{r/2}^{n-1}(x'_0)}|\nabla'f|^2\,dx'\leq C_3 {\mathbf e}(x_0,r,e_n).
$$
\end{theorem}
Finally we shall need the following reverse Poincar\'{e} inequality which can be proved exactly as in the case of $\Lambda$-minimizers of the perimeter (see \cite[Theorem 24.1]{Ma} ).
\begin{theorem}[Reverse Poincar\'{e}] 
Let $(u,E)$ be a $\Lambda$-minimizer of $\mathcal{F}$ in $B_r(x_0)$.
There exist two positive constants $\varepsilon_4$ and $C_4$ such that if $x_0\in \partial E$ and
${\mathbf e}(x_0,r,\nu)<\varepsilon_4$ then
$$
{\mathbf e}(x_0,r/2,\nu)\leq C_4 \biggl(\frac{1}{r^{n+1}}\int_{\partial E\cap B_r(x_0)}|\left<\nu,x-x_0\right>-c|^2d\mathcal{H}^{n-1}+\mathcal{D}(x_0,r)+r\biggr),
$$
for every $c\in \mathbbm R$.
\end{theorem}

\subsection{Weak Euler-Lagrange equation}

The last ingredient to prove the excess improvement is the following Euler-Lagrange equation that we state for $\Lambda r$-minimizers of the rescaled functional $\mathcal{F}_r$ defined below.
For the sake of simplicity we will denote $A_1(x,s)$ the matrix whose entries are $a_{hk}(x,s)$, $A_2(x,s)$ the vector of components $a_{h}(x,s)$, $A_3(x,s)=a(x,s)$ and similarly for $B_i$, $i=1,2,3$.
Accordingly we define
\begin{equation*}
\begin{split}
&\mathcal{F}_r(w,D)
:=\int_{B_1}\big[ F_r(x,u,\D u)+\mathbbm{1}_DG_r(x,u,\D u) \big]\,dx\\
& = \int_{B_1}\big[ \langle (A_{1r}+\mathbbm{1}_D B_{1r})\D w,\D w \rangle + \sqrt{r}\langle A_{2r}+\mathbbm{1}_D B_{2r},\D w \rangle + r(A_{3r}+\mathbbm{1}_DB_{3r}) \big]\,dx,
\end{split}
\end{equation*}
where $r>0$, $x_0\in\Omega$,  $A_{ir}:=A_i (x_0+ry,\sqrt{r} w)$, $B_{ir}:=B_i (x_0+ry,\sqrt{r} w)$, for $i=1,2,3$.
The argument used to prove the next result is similar to the one in \cite[Theorem 7.35]{AFP}.
\begin{theorem}[Weak Euler-Lagrange equation]
Let $(u,E)$ be a $\Lambda r$-minimizer of $\mathcal{F}_r$ in $B_1$. For every vector field $X\in C_0^1(B_1,\R^n)$ and for some constant $\overline{C}=\overline{C}(N,L_D,\sup|X|,\sup|\D X|)>0$ it holds
\begin{equation}
\begin{split}
\label{E-L}
\int_{\dd E} \textnormal{div}_\tau X\,d\mathcal{H}^{n-1}\leq \overline{C}\int_{B_1}\big( |\D u|^2+r \big)\,dx+\Lambda r\int_{\dd E}|X|\,d\mathcal{H}^{n-1},
\end{split}
\end{equation}
where $L_D$ is the greatest Lipschitz constant of the data 
$a_{ij},b_{ij}, a_i,b_i,a,b$, and $\textnormal{div}_\tau$ denotes the tangential divergence on $\dd E$.
\end{theorem}

\begin{proof}
Let us fix $X\in C_0^1(B_1,\R^n)$. We set $\Phi_t(x)=x+tX(x)$, for any $t>0$, $E_t=\Phi_t(E)$ and $u_t=u\circ\Phi_t^{-1}$. From the $\Lambda r$-minimality it follows
\begin{equation}
\label{eqq5}
\begin{split}
& [P(E_t,B_1)-P(E,B_1)]+\Lambda r|E_t\Delta E|\\
& +\int_{B_1}[F_r(y,u_t,\D u_t)+\mathbbm{1}_{E_t}(y)G_r(y,u_t,\D u_t)]\,dy\\
& -\int_{B_1}[F_r(x,u,\D u)+\mathbbm{1}_{E}(x)G_r(x,u,\D u)]\,dx\geq 0.
\end{split}
\end{equation}
In order to obtain \eqref{E-L} we will divide by $t$ and pass to the upper limit as $t\rightarrow 0^+$. Let us study these terms separately. The first variation of the area gives
\begin{equation}
\label{eqq2}
\lim_{t\rightarrow 0^+} \frac{1}{t}[P(E_t,B_1)-P(E,B_1)]=\int_{\dd E} \textnormal{div}_\tau X\,d\mathcal{H}^{n-1}.
\end{equation}
We can deal with the second term observing that
\begin{equation}
\label{eqq3}
\lim_{t\rightarrow 0^+}\frac{\abs{E_t\Delta E}}{t}\leq \int_{\dd E}\abs{X\cdot\nu_E}\,d\mathcal{H}^{n-1},
\end{equation}
(see for instance \cite[Theorem 3.2]{JP}). In the first bulk term we make the change of variables $y=\Phi_t(x)$ with $x\in B_1$ and $t>0$, taking into account that
\begin{equation*}
\D \Phi_t^{-1}(\Phi_t(x))=I-t\D X(x)+o(t), \quad \textnormal{J}\Phi_t(x)=1+t\textnormal{div}X(x)+o(t).
\end{equation*}
Thus we gain
\begin{equation*}
\begin{split}
& \int_{B_1}\big[F_r(y,u_t,\D u_t)+\mathbbm{1}_{E_t}(y)G_r(y,u_t,\D u_t)\big]\,dy\\
& =\int_{B_1} \big[ F_r(\Phi_t(x),u,\D u)+\mathbbm{1}_E(x) G_r(\Phi_t(x),u,\D u) \big](1+t\textnormal{div}X)\,dx\\
& - t\int_{B_1}\big[ 2\big\langle C_1\D u\D X,\D u \big\rangle +\sqrt{r}\big\langle C_2,\D u\D X \big\rangle\big]\,dx+o(t),
\end{split}
\end{equation*}
where we set
\begin{equation*}
C_i:=\tilde{A}_{ir}+\mathbbm{1}_{E}\tilde{B}_{ir}=A_{ir}(\Phi_t(x),u)+\mathbbm{1}_{E}(x)B_{ir}(\Phi_t(x),u),
\end{equation*}
for $i=1,2,3$. By simple calculations we obtain
\begin{equation}
\label{eqq1}
\begin{split}
& \int_{B_1}\big[F_r(y,u_t,\D u_t)+\mathbbm{1}_{E_t}(y)G_r(y,u_t,\D u_t)\big]\,dy\\
& -\int_{B_1}\big[F_r(x,u,\D u)+\mathbbm{1}_{E}(x)G_r(x,u,\D u)\big]\,dx\\
& = \int_{B_1} \big\{ F_r(\Phi_t(x),u,\D u)+\mathbbm{1}_E(x) G_r(\Phi_t(x),u,\D u)-[F_r(x,u,\D u)+\mathbbm{1}_{E}(x)G_r(x,u,\D u)] \big\}\,dx\\
& +t\bigg[ \int_{B_1} \big[F_r(\Phi_t(x),u,\D u)+\mathbbm{1}_E(x) G_r(\Phi_t(x),u,\D u)\big]\textnormal{div}X\,dx\\
& - \int_{B_1}\big[ 2\big\langle C_1\D u\D X,\D u \big\rangle + \sqrt{r}\big\langle C_2,\D u\D X \big\rangle\big]\,dx\bigg]+o(t).
\end{split}
\end{equation}
Let us estimate the first of the three terms. By Lipschitz continuity and Young's inequality we get
\begin{equation*}
\begin{split}
& \int_{B_1} \Big\{ \big\langle F_r(\Phi_t(x),u,\D u)+\mathbbm{1}_E(x) G_r(\Phi_t(x),u,\D u)-[F_r(x,u,\D u)+\mathbbm{1}_{E}(x)G_r(x,u,\D u)] \Big\}\,dx\\
& \leq c(L_D)t\int_{B_1}|X|[|\D u|^2+\sqrt{r}|\D u|+r]\,dx \leq c(L_D)t\int_{B_1}|X|[|\D u|^2+r]\,dx.
\end{split}
\end{equation*}
Finally, dividing by $t$ and passing to the upper limit as $t\rightarrow 0^+$ we infer
\begin{equation}
\label{eqq4}
\begin{split}
& \limsup_{t\rightarrow 0^+}\frac{1}{t}\bigg[ \int_{B_1}[F_r(y,u_t,\D u_t)+\mathbbm{1}_{E_t}(y)G_r(y,u_t,\D u_t)]\,dy\\
& -\int_{B_1}[F_r(x,u,\D u)+\mathbbm{1}_{E}G_r(x,u,\D u)]\,dx \bigg]\\
&\leq c(L_D)\int_{B_1}|X|[|\D u|^2+r]\,dx+\int_{B_1}[F_r(x,u,\D u)+\mathbbm{1}_{E}G_r(x,u,\D u)]\textnormal{div}X\,dx\\
& -\int_{B_1}\big[ 2\big\langle (A_{1r}+\mathbbm{1}_{E}B_{1r})\D u\D X,\D u \big\rangle + \sqrt{r}\big\langle (A_{2r}+\mathbbm{1}_{E}B_{2r}),\D u\D X \big\rangle \big]\,dx.
\end{split}
\end{equation}
Passing to the upper limit as $t\rightarrow 0^+$ and putting \eqref{eqq2}, \eqref{eqq3}, \eqref{eqq4} together we get
\begin{equation*}
\begin{split}
& \int_{\dd E} \textnormal{div}_\tau X\,d\mathcal{H}^{n-1}\\
& \leq c(L_D)\int_{B_1}|X|[|\D u|^2+r]\,dx+\bigg|\int_{B_1}[F_r(x,u,\D u)+\mathbbm{1}_{E}G_r(x,u,\D u)]\textnormal{div}X\,dx\bigg|\\
& +\int_{B_1}\big| 2\big\langle (A_{1r}+\mathbbm{1}_{E}B_{1r})\D u\D X,\D u \big\rangle + \sqrt{r}\big\langle (A_{2r}+\mathbbm{1}_{E}B_{2r}),\D u\D X \big\rangle \big|\,dx+\Lambda r\int_{\dd E}\abs{X}\,d\mathcal{H}^{n-1}\\
& \leq \overline{C}\int_{B_1}\big( |\D u|^2+r \big)\,dx+\Lambda r\int_{\dd E}|X|\,d\mathcal{H}^{n-1},
\end{split}
\end{equation*}
where $\overline{C}=\overline{C}(N,L_D,\sup|X|,\sup|\D X|)$.

\end{proof}

\section{Excess improvement}

\begin{theorem}[Excess improvement]
\label{Miglioramento eccesso}
For every $\tau\in \big(0,\frac{1}{2}\big)$ and $M>0$ there exists a constant $\varepsilon_5=\varepsilon_5(\tau,M)\in(0,1)$ such that if $(u,E)$ is a $\Lambda$-minimizer of $\mathcal{F}$ in $B_r(x_0)$ with $x_0\in\dd E$ and
\begin{equation}
\label{3}
{\mathbf e}(x_0,r)\leq\varepsilon_5 \quad\text{and}\quad \mathcal{D}(x_0,r)+r\leq M{\mathbf e}(x_0,r)
\end{equation}
then there exists a positive constant $C_5$, depending on $\norm{\D u}_{L^2(B_r(x_0))}$, such that
\begin{equation*}
{\mathbf e}(x_0,\tau r)\leq C_5(\tau^2{\mathbf e}(x_0,r)+\mathcal{D}(x_0,2\tau r)+\tau r).
\end{equation*}
\end{theorem}
\begin{proof}
Without loss of generality we may assume that $\tau<\frac{1}{8}$. Let us rescale 
and assume by contradiction that there exists an infinitesimal sequence $\lbrace\varepsilon_h\rbrace_{h\in\N}\subseteq\R^+$, a sequence $\lbrace r_h\rbrace_{h\in\N}\subseteq\R^+$ and a sequence $\lbrace (u_h,E_h) \rbrace_{h\in\N}$ of $\Lambda r_h$-minimizers of $\mathcal{F}_{r_h}$ in $B_1$, with equibounded energies, such that, denoting by ${\mathbf e}_h$ the excess of $E_h$ and by $\mathcal{D}_h$  the rescaled Dirichlet integral of $u_h$, we have
\begin{equation*}
{\mathbf e}_h(0,1)=\varepsilon_h, \quad \mathcal{D}_h(0,1)+r_h\leq M\varepsilon_h
\end{equation*}
and
\begin{equation*}
{\mathbf e}_h(0,\tau)>C_5(\tau^2{\mathbf e}(0,1)+\mathcal{D}(0,2\tau)+\tau r_h),
\end{equation*}
with some positive constant $C_5$ to be chosen.
Up to rotating each $E_h$ we may also assume that for all $h\in\N$
\begin{equation*}
{\mathbf e}_h(0,1)=\frac{1}{2}\int_{\dd E_h\cap B_1}\abs{\nu_{E_h}-e_n}^2\, d\mathcal{H}^{n-1}.
\end{equation*}
\textbf{Step 1.} Thanks to the Lipschitz approximation theorem, for $h$ sufficiently large, there exists a 1-Lipschitz function $f_h\colon\R^{n-1}\rightarrow\R$ such that
\begin{equation}
\label{1}
\sup_{\R^{n-1}}\abs{f_h}\leq C_3\varepsilon_h^{\frac{1}{2(n-1)}}, \quad \mathcal{H}^{n-1}((\dd E_h\Delta\Gamma_{f_h})\cap B_{\frac{1}{2}})\leq C_3\varepsilon_h, \quad \int_{B_{\frac{1}{2}}^{n-1}}\abs{\D' f_h}^2\,dx'\leq C_3\varepsilon_h.
\end{equation}
We define
\begin{equation*}
g_h(x'):=\frac{f_h(x')-a_h}{\sqrt{\varepsilon_h}}, \quad\text{where}\quad a_h=\int_{B_{\frac{1}{2}}^{n-1}}f_h\,dx'
\end{equation*}
and we assume up to a subsequence that $\lbrace g_h \rbrace_{h\in\N}$ converges weakly in $H^1(B_{\frac{1}{2}}^{n-1})$ and strongly in $L^2(B_{\frac{1}{2}}^{n-1})$ to a function $g$.\\
We prove that $g$ is harmonic in $B_{\frac{1}{2}}^{n-1}$. It's enough show that
\begin{equation}
\label{2}
\lim_{h\rightarrow +\infty}\frac{1}{\sqrt{\varepsilon_h}}\int_{B_{\frac{1}{2}}^{n-1}}\frac{\langle \D' f_h,\D'\phi\rangle}{\sqrt{1+\abs{\D' f_h}^2}}\,dx'=0,
\end{equation}
for all $\phi\in C_0^1(B_{\frac{1}{2}}^{n-1})$; indeed, if $\phi\in C_0^1(B_{\frac{1}{2}}^{n-1})$, by weak convergence we have
\begin{equation*}
\begin{split}
& \int_{B_{\frac{1}{2}}^{n-1}} \langle\D' g,\D'\phi\rangle\,dx' =
\lim_{h\rightarrow +\infty}\frac{1}{\sqrt{\varepsilon_h}}\int_{B_{\frac{1}{2}}^{n-1}}\langle\D' f_h,\D'\phi\rangle\,dx'\\
& = \lim_{h\rightarrow +\infty}\frac{1}{\sqrt{\varepsilon_h}}\bigg\lbrace \int_{B_{\frac{1}{2}}^{n-1}}\frac{\langle \D 'f_h,\D'\phi\rangle}{\sqrt{1+\abs{\D' f_h}^2}}\,dx'+\int_{B_{\frac{1}{2}}^{n-1}} \bigg[ \langle \D' f_h,\D'\phi\rangle - \frac{\langle \D' f_h,\D'\phi\rangle}{\sqrt{1+\abs{\D' f_h}^2}}\bigg]\,dx' \bigg\rbrace.
\end{split}
\end{equation*}
Using the lipschitz-continuity of $f_h$ and the third equation in \eqref{1} we infer that the second term in the previous equality is infinitesimal:
\begin{equation*}
\begin{split}
& \limsup_{h\rightarrow +\infty}\frac{1}{\sqrt{\varepsilon_h}}\bigg |\int_{B_{\frac{1}{2}}^{n-1}} \bigg [ \langle \D' f_h,\D'\phi\rangle - \frac{\langle \D' f_h,\D'\phi\rangle}{\sqrt{1+\abs{\D' f_h}^2}}\bigg]\,dx'\bigg |\\
& \leq \limsup_{h\rightarrow +\infty}\frac{1}{\sqrt{\varepsilon_h}} \int_{B_{\frac{1}{2}}^{n-1}} \abs{\D' f_h}\abs{\D'\phi}\frac{\sqrt{1+\abs{\D' f_h}^2}-1}{\sqrt{1+\abs{\D' f_h}^2}}\,dx'\\
& \leq \limsup_{h\rightarrow +\infty}\frac{1}{2\sqrt{\varepsilon_h}} \int_{B_{\frac{1}{2}}^{n-1}} \abs{\D'\phi}\abs{\D' f_h}^2\,dx'\leq \lim_{h\rightarrow +\infty}\frac{C_3\norm{\D'\phi}_\infty\sqrt{\varepsilon_h}}{2}=0.
\end{split}
\end{equation*}
Therefore, we can prove \eqref{2}. We fix $\delta>0$ so that spt$\phi\times [-2\delta,2\delta]\subseteq B_{\frac{1}{2}}$, choose a cut-off function $\psi\colon\R\rightarrow[0,1]$ with spt$\psi\subseteq (-2\delta,2\delta)$, $\psi=1$ in $(-\delta,\delta)$ and apply to $E_h$ the weak Euler-Lagrange equation with $X=\phi\psi e_n$. By the height bound, for $h$ sufficiently large it holds that $\dd E_h\cap B_{\frac{1}{2}}\subseteq B_{\frac{1}{2}}^{n-1}\times (-\delta,\delta)$. Plugging $X$ in the weak Euler-Lagrange equation and using the assumption in \eqref{3}, we have
\begin{equation*}
\begin{split}
&-\gamma\int_{\dd E_h\cap B_{\frac{1}{2}}}\langle \nu_{E_h}, e_n\rangle \langle \D'\phi,\nu_{E_h}' \rangle\,d\mathcal{H}^{n-1}\\
& \leq c(N,L_D,\phi,\psi)\int_{B_{\frac{1}{2}}} \big(|\D u_h|^2+r_h\big)\,dx+\Lambda r_h\int_{\dd E_h\cap B_{\frac{1}{2}}}\abs{\phi\psi}\,d\mathcal{H}^{n-1}\\
& \leq c(n,N,\Lambda,L_D,\phi,\psi)\varepsilon_h.
\end{split}
\end{equation*}
Therefore, if we replace $\phi$ by $-\phi$, we infer
\begin{equation}
\label{5}
\lim_{h\rightarrow +\infty}\frac{1}{\sqrt{\varepsilon_h}} \int_{\dd E_h\cap B_{\frac{1}{2}}}\langle \nu_{E_h}, e_n\rangle \langle \D'\phi,\nu_{E_h}' \rangle\,d\mathcal{H}^{n-1}=0.
\end{equation}
Decomposing $\dd E_h\cap B_{\frac{1}{2}}=\big([\Gamma_{f_h}\cup(\dd E_h\setminus\Gamma_{f_h})]\setminus(\Gamma_{f_h}\setminus\dd E_h)\big)\cap B_{\frac{1}{2}}$, we deduce
\begin{equation}
\label{4}
\begin{split}
& -\frac{1}{\sqrt{\varepsilon_h}}\int_{\dd E_h\cap B_{\frac{1}{2}}}\langle \nu_{E_h}, e_n\rangle \langle \D'\phi,\nu_{E_h}' \rangle\,d\mathcal{H}^{n-1}
 =\frac{1}{\sqrt{\varepsilon_h}}\bigg[-\int_{ \Gamma_{f_h}\cap B_{\frac{1}{2}}}\langle \nu_{E_h}, e_n\rangle \langle \D'\phi,\nu_{E_h}' \rangle\,d\mathcal{H}^{n-1}\\
& - \int_{(\dd E_h\setminus \Gamma_{f_h})\cap B_{\frac{1}{2}}}\langle \nu_{E_h}, e_n\rangle \langle \D'\phi,\nu_{E_h}' \rangle\,d\mathcal{H}^{n-1}+\int_{(\Gamma_{f_h}\setminus \dd E_h)\cap B_{\frac{1}{2}}}\langle \nu_{E_h}, e_n\rangle \langle \D'\phi,\nu_{E_h}' \rangle\,d\mathcal{H}^{n-1}\bigg].
\end{split}
\end{equation}
Since by the second inequality in \eqref{1} we have
\begin{equation*}
\bigg |\frac{1}{\sqrt{\varepsilon_h}} \int_{(\dd E_h\setminus \Gamma_{f_h})\cap B_{\frac{1}{2}}}\langle \nu_{E_h}, e_n\rangle \langle \D'\phi,\nu_{E_h}' \rangle\,d\mathcal{H}^{n-1} \bigg | \leq C_3\sqrt{\varepsilon_h}\sup_{\R^{n-1}}\abs{\D'\phi},
\end{equation*}
\begin{equation*}
\bigg |\frac{1}{\sqrt{\varepsilon_h}} \int_{(\Gamma_{f_h}\setminus \dd E_h)\cap B_{\frac{1}{2}}}\langle \nu_{E_h}, e_n\rangle \langle \D'\phi,\nu_{E_h}' \rangle\,d\mathcal{H}^{n-1} \bigg | \leq C_3\sqrt{\varepsilon_h}\sup_{\R^{n-1}}\abs{\D'\phi},
\end{equation*}
then by \eqref{5} and the area formula, we infer
\begin{equation*}
0=\lim_{h\rightarrow +\infty}\frac{-1}{\sqrt{\varepsilon_h}}\int_{ \Gamma_{f_h}\cap B_{\frac{1}{2}}}\langle \nu_{E_h}, e_n\rangle \langle \D'\phi,\nu_{E_h}' \rangle\,d\mathcal{H}^{n-1}=\lim_{h\rightarrow +\infty}\frac{1}{\sqrt{\varepsilon_h}}\int_{B_{\frac{1}{2}}^{n-1}}\frac{\langle \D' f_h,\D'\phi\rangle}{\sqrt{1+\abs{\D' f_h}^2}}\,dx'.
\end{equation*}
This proves that $g$ is harmonic.\\
\textbf{Step 2.} The proof of this step now follows  exactly as in [FJ] using the eight bound lemma and the revers Poincar\'{e} inequality. We give it here to be thorough. By the mean value property of harmonic functions, Lemma 25.1 in \cite{Ma}, Jensen inequality, semicontinuity and the third inequality in \eqref{1} we deduce that
\begin{equation*}
\begin{split}
& \lim_{h\rightarrow +\infty}\frac{1}{\varepsilon_h}\int_{B_{2\tau}^{n-1}}\abs{f_h(x')-(f_h)_{2\tau}-\langle(\D' f_h)_{2\tau},x'\rangle}^2\,dx'\\
& =\int_{B_{2\tau}^{n-1}}\abs{g(x')-(g)_{2\tau}-\langle(\D' g)_{2\tau},x'\rangle}^2\,dx'\\
& = \int_{B_{2\tau}^{n-1}}\abs{g(x')-g(0)-\langle\D' g(0),x'\rangle}^2\,dx'\\
&\leq c(n)\tau^{n-1}\sup_{x'\in B_{2\tau}^{n-1}}\abs{g(x')-g(0)-\langle\D' g(0),x'\rangle}^2\\
& \leq c(n)\tau^{n+3}\int_{B_{\frac{1}{2}}^{n-1}}\abs{\D' g}^2\,dx'\leq c(n)\tau^{n+3}\liminf_{h\rightarrow +\infty}\int_{B_{\frac{1}{2}}^{n-1}}\abs{\D' g_h}^2\,dx'\\
& \leq \tilde{C}(n,C_3)\tau^{n+3}.
\end{split}
\end{equation*}
On one hand, using the area formula, the mean value property, the previous inequality and setting
\begin{equation*}
c_h:=\frac{(f_h)_{2\tau}}{\sqrt{1+\abs{(\D' f_h)_{2\tau}}^2}}, \quad \nu_h:=\frac{(-(\D' f_h)_{2\tau},1)}{\sqrt{1+\abs{(\D' f_h)_{2\tau}}^2}},
\end{equation*}
we have
\begin{equation*}
\begin{split}
&\limsup_{h\rightarrow +\infty}\frac{1}{\varepsilon_h}\int_{\dd E_h\cap\Gamma_{f_h}\cap B_{2\tau}}\abs{\langle\nu_h,x\rangle-c_h}^2\,d\mathcal{H}^{n-1}\\
& = \limsup_{h\rightarrow +\infty}\frac{1}{\varepsilon_h}\int_{\dd E_h\cap\Gamma_{f_h}\cap B_{2\tau}}\frac{\abs{\langle -(\D' f_h)_{2\tau},x'\rangle+f_h(x')-(f_h)_{2\tau}}}{1+\abs{(\D' f_h)_{2\tau}}^2}^2\sqrt{1+\abs{\D 'f_h(x')}^2}\,dx'\\
& \leq \lim_{h\rightarrow +\infty}\frac{1}{\varepsilon_h}\int_{B_{2\tau}^{n-1}}\abs{f_h(x')-(f_h)_{2\tau}-\langle(\D' f_h)_{2\tau},x'\rangle}^2\,dx'\leq \tilde{C}(n,C_3)\tau^{n+3}.
\end{split}
\end{equation*}
On the other hand, arguing as in Step 1, we immediately get from the height bound and the first two inequalities in \eqref{1} that
\begin{equation*}
\lim_{h\rightarrow +\infty}\frac{1}{\varepsilon_h}\int_{(\dd E_h\setminus\Gamma_{f_h})\cap B_{2\tau}}\abs{\langle\nu_h,x\rangle-c_h}^2\,d\mathcal{H}^{n-1}=0.
\end{equation*}
Hence we conclude that
\begin{equation}
\label{6}
\limsup_{h\rightarrow +\infty}\frac{1}{\varepsilon_h}\int_{\dd E_h\cap B_{2\tau}}\abs{\langle\nu_h,x\rangle-c_h}^2\,d\mathcal{H}^{n-1}\leq \tilde{C}(n,C_3)\tau^{n+3}.
\end{equation}
We claim that the sequence $\lbrace {\mathbf e}_h(0,2\tau,\nu_h) \rbrace_{h\in\N}$ is infinitesimal; indeed, by the definition of excess, Jensen'2 inequality and the third inequality in \eqref{1} we have
\begin{equation*}
\begin{split}
& \limsup_{h\rightarrow +\infty}\int_{\dd E_h\cap B_{2\tau}}\abs{\nu_{E_h}-\nu_h}^2\,d\mathcal{H}^{n-1}\\
& \leq\limsup_{h\rightarrow +\infty}\bigg[ 2\int_{\dd E_h\cap B_{2\tau}}\abs{\nu_{E_h}-e_n}^2\,d\mathcal{H}^{n-1}+2\abs{e_n-\nu_h}^2\mathcal{H}^{n-1}(\dd E_h\cap B_{2\tau})\bigg]\\
& \leq \limsup_{h\rightarrow +\infty}\bigg[4\varepsilon_h+2\mathcal{H}^{n-1}(B_{2\tau})\frac{\abs{((\D' f_h)_{2\tau},\sqrt{1+\abs{(\D' f_h)_{2\tau}}^2}-1)}^2}{1+\abs{(\D' f_h)_{2\tau}}^2}\bigg]\\
& \leq \limsup_{h\rightarrow +\infty}\big[4\varepsilon_h+4\mathcal{H}^{n-1}(B_{2\tau})\abs{(\D' f_h)_{2\tau}}^2\big]\\
& \leq \limsup_{h\rightarrow +\infty}\bigg[ 4\varepsilon_h+4\int_{B_{\frac{1}{2}}^{n-1}}\abs{\D' f_h}^2\,dx' \bigg]\leq \lim_{h\rightarrow +\infty}[4\varepsilon_h+4C_3\varepsilon_h]=0.
\end{split}
\end{equation*}
Therefore applying the reverse Poicaré's inequality and \eqref{6} we have for $h$ large that
\begin{equation*}
\begin{split}
{\mathbf e}_h(0,\tau)\leq {\mathbf e}_h(0,\tau,\nu_h)\leq C_4(\tilde{C}\tau^2{\mathbf e}_h(0,1)+\mathcal{D}(0,2\tau)+2\tau r_h),
\end{split}
\end{equation*}
which is a contradiction if we choose $C_5>C_4\max\lbrace\tilde{C},2\rbrace$.
\end{proof}

\section{Proof of the main theorem}
The proof works exactly as in \cite{FJ}. We give here some details to emphasize the dependence of $\varepsilon$ appearing in the statement of Theorem \ref{Teorema principale} from the structural data of the functional. The proof is divided in four steps.\\
\textbf{Step 1.} For every $\tau\in(0,1)$ there exists $\varepsilon_6=\varepsilon_6(\tau)>0$ such that if ${\mathbf e}(x,r)\leq\varepsilon_6$ then
\begin{equation*}
\mathcal{D}(x,\tau r)\leq C_0\tau\mathcal{D}(x,r),
\end{equation*}
where $C_0$ is from Lemma \ref{Lemma decadimento 1}. Assume by contradiction that for some $\tau\in(0,1)$ there exist two positive sequences $\varepsilon_h$ and $r_h$ and a sequence $(u_h,E_h)$ of $\Lambda r_h-$ minimizers of $\mathcal{F}_{r_h}$ in $B_1$ with equibounded energies, such that, denoting by ${\mathbf e}_h$ the excess of $E_h$ and by $\mathcal{D}_h$ the rescaled Dirichlet integral of $u_{E_h}$, we have that $0\in\dd E_h$,
\begin{equation}
\label{Eqn 7}
{\mathbf e}_h(0,1)=\varepsilon_h\rightarrow 0 \quad\text{and}\quad \mathcal{D}_h(0,\tau)>C_0\tau\mathcal{D}_h(0,1).
\end{equation}
Thanks to the energy upper bound (Theorem \ref{Energy upper bound}) and the compactness lemma (Lemma \ref{Lemma compattezza}) we may assume that $E_h\rightarrow E$ in $L^1(B_1)$ and $0\in\dd E$. Since, by lower semicontinuity, the excess of $E$ at 0 is null, $E$ is a half space in $B_1$, say $H$. In particular, for $h$ large it holds
\begin{equation*}
|(E_h\Delta H)\cap B_1|<\varepsilon_0(\tau)|B_1|,
\end{equation*}
where $\varepsilon_0$ is from Lemma \ref{Lemma decadimento 1}, which gives a contradiction with the inequality \eqref{Eqn 7}.\\
\textbf{Step 2.} Let $U\subset\subset\Omega$ be an open set. Prove that for every $\tau\in(0,1)$ there exist two positive constants $\varepsilon=\varepsilon(\tau,U)$ and $C_6$ such that if $x_0\in\dd E$, $B_r(x_0)\subset U$ and ${\mathbf e}(x_0,r)+\mathcal{D}(x_0,r)+r<\varepsilon$ then
\begin{equation}
\label{Eqn 8}
{\mathbf e}(x_0,\tau r)+\mathcal{D}(x_0,\tau r)+\tau r\leq C_6\tau({\mathbf e}(x_0,r)+\mathcal{D}(x_0,r)+r).
\end{equation}
Fix $\tau\in(0,1)$ and assume without loss of generality that $\tau<\frac{1}{2}$ We can distinguish two cases. \\
\textit{Case 1:} $\mathcal{D}(x_0,r)+r\leq \tau^{-n}{\mathbf e}(x_0,r)$. If ${\mathbf e}(x_0,r)<\min\{ \varepsilon_5 (\tau,\tau^{-n}),\varepsilon_6(2\tau)\}$ it follows from Theorem \ref{Miglioramento eccesso} and Step 1 that
\begin{equation*}
\begin{split}
{\mathbf e}(x_0,\tau r)\leq C_5(\tau^2{\mathbf e}(x_0,r)+\mathcal{D}(x_0,2\tau r)+\tau r)\leq C_5(\tau^2{\mathbf e}(x_0,r)+2C_0\mathcal{D}(x_0,\tau r)+\tau r)
\end{split}.
\end{equation*}
\textit{Case 2:} ${\mathbf e}(x_0,r)\leq\tau^{n}(\mathcal{D}(x_0,r)+r)$. By the property of the excess at different scales, we infer
\begin{equation*}
{\mathbf e}(x_0,\tau r)\leq \tau^{1-n}{\mathbf e}(x_0,r)\leq \tau^{n}(\mathcal{D}(x_0,r)+r).
\end{equation*}
We conclude that choosing $\varepsilon=\min\{ \varepsilon_5(\tau,\tau^{-n}),\varepsilon_6(2\tau),\varepsilon_6(\tau) \}$, inequality \eqref{Eqn 8} is verified.\\
\textbf{Step 3.} Fix $\sigma\in(0,\frac{1}{2})$ and choose $\tau_0\in(0,1)$ such that $C_6\tau_0\leq\tau_0^{2\sigma}$. Let $U\subset\subset\Omega$ be an open set. We define
\begin{equation*}
\begin{split}
\Gamma\cap U:=
&\{ x\in\dd E\cap U\,:\, {\mathbf e}(x,r)+\mathcal{D}(x,r)+r<\varepsilon(\tau_0,U)\\
& \text{ for some }r>0\text{ such that }B_r(x_0)\subset U \}.
\end{split}
\end{equation*}
Note that $\Gamma\cap U$ is relatively open in $\dd E$. We show that $\Gamma\cap U$ is a $C^{1,\sigma}$-hypersurface. Indeed inequality \eqref{Eqn 8} implies via standard iteration argument that if $x_0\in\Gamma\cap U$ there exist $r_0>0$ and a neighborhood $V$ of $x_0$ such that for every $x\in\dd E\cap V$ it holds
\begin{equation*}
{\mathbf e}(x,\tau_0^k r_0)+\mathcal{D}(x,\tau_0^k r_0)+\tau_0^k r_0\leq \tau_0^{2\sigma k}, \quad\text{for }k\in\N_0.
\end{equation*}
In particular ${\mathbf e}(x,\tau_0^k r_0)\leq \tau_0^{2\sigma k}$ and, arguing as in \cite{FJ}, we obtain that for every $x\in\dd E\cap V$ and $0<s<t<r_0$ it holds
\begin{equation*}
|(\nu_E)_s(x)-(\nu_E)_t(x)|\leq ct^\sigma,
\end{equation*}
for some constant $c=c(n,\tau_0,r_0)$, where
\begin{equation*}
(\nu_E)_t(x)=\fint_{\dd E\cap B_t(x)}\nu_E(y)\,d\mathcal{H}^{n-1}.
\end{equation*}
The previous estimate first implies that $\Gamma\cap U$ is $C^1$. By a standard argument we then deduce again from the same estimate that $\Gamma\cap U$ is a $C^{1,\sigma}$-hypersurface. Finally we define $\Gamma:=\cup_i(\Gamma\cap U_i)$, where $(U_i)_i$ is an increasing sequence of open sets such that $U_i\subset\subset\Omega$ and $\Omega=\cup_i U_i$.\\
\textbf{Step 4.} Finally we prove that there exists $\eta>0$ such that
\begin{equation*}
\mathcal{H}^{n-1-\eta}(\dd E\setminus \Gamma)=0.
\end{equation*}
Setting $\Sigma=\{ x\in\dd E\setminus\Gamma\,:\, \lim\limits_{r\rightarrow 0}\mathcal{D}(x,r)=0 \}$, by Lemma \ref{Lemma maggiore sommabilità}, $\D u\in L^{2s}_{loc}(\Omega)$ for some $s>1$, depending only on $\nu,N,L,n$, and we have that
\begin{equation*}
\text{dim}_{\mathcal{H}}\big(\{x\in\Omega\,:\, \limsup_{r\rightarrow 0}\mathcal{D}(x,r)>0 \}\big)\leq n-s.
\end{equation*}
The conclusion follows as in \cite{FJ} showing that $\Sigma=\emptyset$ when $n\leq 7$ and $\text{dim}_\mathcal{H}(\Sigma)\leq n-8$ for $n\geq 8$.



\end{document}